%% file: m.tex
\documentclass[reqno,12pt]{amsart}

\include{environ}
\include{symbols}

\begin{document}

\title[The Navier-Stokes-Voight Model for Image Inpainting]
      {The Navier-Stokes-Voight Model for Image Inpainting}

\author[M.A. Ebrahimi]{M.A. Ebrahimi}
\email{maebrahi@math.ucsd.edu}

\author[M. Holst]{Michael Holst}
\email{mholst@math.ucsd.edu}
\thanks{MH was supported in part by NSF Awards~0715146 and 0915220.}

\author[E. Lunasin]{Evelyn Lunasin}
\email{elunasin@math.ucsd.edu}
\thanks{ME and EL were supported in part by NSF Award~0715146.}

\address{Department of Mathematics\\
         University of California San Diego\\ 
         La Jolla CA 92093}

\date{\today}


\input{abs}

\maketitle

\vspace*{-1.10cm}
{\footnotesize
\tableofcontents
}

\input{setup}

\input{numerics}
\input{theory}

\input{summary}

\input{ack}

\bibliographystyle{siam}
\bibliography{../bib/papers,../bib/books}

\end{document}

%% file: environ.tex
\pdfoutput=1

\usepackage{color} % black,white,red,green,blue,cyan,magenta,yellow
\usepackage{graphicx}
\usepackage{pstricks}
\usepackage{ifpdf}
\ifpdf
    \usepackage{hyperref}
    \hypersetup{
        bookmarks=true,         % show bookmarks bar?
        unicode=false,          % non-Latin characters in Acrobat’s bookmarks
        pdftoolbar=true,        % show Acrobat’s toolbar?
        pdfmenubar=true,        % show Acrobat’s menu?
        pdffitwindow=false,     % window fit to page when opened
        pdfstartview={FitH},    % fits the width of the page to the window
        pdftitle={MCP Article},      % title
        pdfauthor={Michael Holst},   % author
        pdfsubject={Mathematics},    % subject of the document
        pdfcreator={Michael Holst},  % creator of the document
        pdfproducer={Michael Holst}, % producer of the document
        pdfkeywords={PDE, analysis, mathematical physics}, % list of keywords
        pdfnewwindow=true,      % links in new window
        colorlinks=true,        % false: boxed links; true: colored links
        linkcolor=red,          % color of internal links
        citecolor=blue,         % color of links to bibliography
        filecolor=magenta,      % color of file links
        urlcolor=cyan           % color of external links
    }

\else
\fi

\usepackage{times}
\usepackage{amsfonts}

\usepackage{amsmath}
\usepackage{amsthm}
\usepackage{amssymb}
\usepackage{amsbsy}
\usepackage{amscd}

\usepackage{verbatim}
\usepackage{subfigure}

\usepackage{multicol}        % used for the two-column index
\usepackage[bottom]{footmisc}% places footnotes at page bottom

\usepackage{setspace}
\newtheorem{theorem}{Theorem}[section]

\newtheorem{lemma}[theorem]{Lemma}

\numberwithin{equation}{section}  %amsmath command: tie counter to section

  \newcounter{mnote}
  \let\oldmarginpar\marginpar
    \renewcommand\marginpar[1]{\-\oldmarginpar[\raggedleft\footnotesize #1]%
    {\raggedright\footnotesize #1}}
\definecolor{myblue}{rgb}{0.2,0.2,0.7}
\definecolor{mygreen}{rgb}{0,0.6,0}
\definecolor{mycyan}{rgb}{0,0.6,0.6}
\definecolor{myred}{rgb}{0.9,0.2,0.2}
\definecolor{mymagenta}{rgb}{0.9,0.2,0.9}
\definecolor{mywhite}{rgb}{1.0,1.0,1.0}
\definecolor{myblack}{rgb}{0.0,0.0,0.0}

%% file: symbols.tex
\newcommand{\beq}{\begin{equation}}
\newcommand{\eeq}{\end{equation}}
\newcommand{\beqa}{\begin{eqnarray}}
\newcommand{\eeqa}{\end{eqnarray}}
\newcommand{\N}{{\mathbb N}}       % Natural numbers.
\newcommand*{\pref}[1]{(\ref{#1})}
\newcommand*{\uh}[1]{u_h^{#1}}
\newcommand*{\Vnorm}[1]{\| #1\|_h}
\newcommand*{\paren}[1]{\left(#1\right)}

\def\aa{\alpha}

\def\ss{\sigma}

\def\Dd{\Delta}

\def\Om{\Omega}

\def\pp{\partial}

%% file: abs.tex
\begin{abstract}
In this paper we investigate the use of the 2D Navier-Stokes-Voight (NSV) model for use in algorithms and explore its limits in the context of image inpainting.
We begin by giving a brief review of the work of Bertalmio {\it et. al.} in 2001 on exploiting an analogy between the image intensity function for the image inpainting problem and the stream function in 2D incompressible fluid.
An approximate solution to the inpainting problem was then obtained by numerically approximating the steady state solution of the 2D Navier-Stokes vorticity transport equation, and simultaneously solving the Poisson problem between the vorticity and stream function, in the region to be inpainted.
This elegant approach allows one to produce an approximate solution to the image inpainting problem by using techniques from computational fluid dynamics (CFD).
Recently, the three-dimensional (3D) Navier-Stokes-Voight (NSV) model of viscoelastic fluid, was suggested by Cao, {\it et. al.} as an inviscid regularization to the 3D Navier-Stokes equations.
We give some background on the NSV mathematical model, describe why it is a good candidate sub-grid-scale turbulence model, and propose this model as an alternative for image inpainting.
We describe an implementation of inpainting use the NSV model, and present numerical results comparing the resulting images when using the NSE and NSV for inpainting.
Our results show that the NSV model allows for a larger time step to converge to the steady state solution, yielding a more efficient numerical process when automating the inpainting process.
We compare quality of the resulting images using subjective measure (human evaluation) and objected measure (by calculating the peak signal-to-noise ratio (PSNR), also known as peak signal-to-reconstructed measure).
We also present some new theoretical results based on energy methods comparing the sufficient conditions for stability of the discretization scheme for the two model equations.
These theoretical and numerical studies shed some light on what can be expected from this category of approach when automating the inpainting problem.
\end{abstract}

%% file: setup.tex
%%%%%%%%%%%%%%%%%%%%%%%%%%%%%%%%%%%%%%%%%%%%%%%%%%%%%%%%%%%%%%%%%%%%%%%%%%%%%%
\section{Introduction}\label{intro}

Image inpainting ($\it{inpainting}$ for short) is the process of correcting a damaged image by filling in the missing or altered data of an image with a better suited data for that region.
The idea is to fill in the damaged part of an image using the information from surrounding areas.
The goal of researchers in this area is to develop algorithms for automatic digital inpainting which mimics the basic manual techniques used by a professional restorer.
In 2001, Bertalmio, {\it et. al.} built a method based on an analogy between image intensity and the stream function in a two-dimensional (2D) incompressible fluid.
The solution to the inpainting problem is then obtained by numerically approximating the steady state solution of the 2D Navier-Stokes vorticity transport equation, for some small viscosity, and simultaneously solving the Poisson problem between the vorticity and stream function in the region to be inpainted.
This approach enables automation of the inpainting process by which inpainting is done using techniques from computational fluid dynamics (CFD).
However, difficulties which arise in CFD are also inherited.

For inpainting, the viscosity in the fluid model should be as small as possible to preserve edges.
However, CFD simulation of high Reynolds number flows (flows with very small viscosity) requires very fine mesh in order to resolve the wide range of scales of motion contributing to the dynamics of the flow.
Sub-grid scale modeling is an alternative that reduces the computational requirements when simulating turbulent flows.
Recently, the three-dimensional (3D) Navier-Stokes-Voight (NSV) model of viscoelastic fluid,   
\begin{equation}
\aligned
-\alpha^2\Delta u_t+u_t -\nu \Delta u  + (u\cdot\nabla)u +\nabla p &=f, \\
  \nabla\cdot u&=0,
\endaligned
\label{nsv}
\end{equation}
with initial condition $u(x,0) = u^{in}(x)$ was suggested by Cao, {\it et. al.} as an inviscid regularization to the 3D Navier-Stokes equations (NSE), where the length-scale $\alpha$ is considered as the regularizing parameter, $u$ is the velocity of the fluid with viscosity $\nu>0$, $p$ is the pressure and $f$ is the body force.
Note that when $\alpha=0$, we recover the Navier-Stokes equations of motion.
The system (\ref{nsv}) was first introduced in 1973 by Oskolkov (see~\cite{Osk73,Osk80}) as a model of a motion of linear, viscoelastic fluid.
In that setting, $\alpha$ is thought of as a length-scale parameter characterizing the elasticity of the fluid.
The 3D NSV model is shown to be globally well-posed~\cite{Osk73,Osk80} and has a finite-dimensional global attractor~\cite{KT07}, making it an attractive sub-grid scale turbulence model for numerical simulation.
In the presence of physical boundaries, the above regularization of the NSE is different in nature from other alpha regularization models in~\cite{CHT05,CLT06,CFHOT98,CFHOT99a,CFHOT99b,CHOT05,ILT06,LL06,OlTi07} (see also Section \ref{NSV-section}), because it does not require additional boundary conditions.

In this paper, we investigate a new approach which arises by combining the recent technique in~\cite{BBS01}, which uses ideas from classical fluid dynamics to propagate isophote lines, and recent results in~\cite{CLT06,KLT07,KT07}, which studies the NSV model of viscoelastic fluid as a candidate sub-grid scale turbulence model for purposes of numerical simulation.
Other PDEs have been proposed recently for example in~\cite{BoMarz07,TscDer05} which are possibly considered to be more efficient and perhaps easier to implement.
Of particular interest to us in this area is to explore different sub-grid scale turbulent models applied to image inpainting.
For the reasons outlined above, we have chosen this particular turbulence model to study in the context of inpainting.
Using this new model, we will study the effect of the regularization parameter $\alpha$ (which one can think of as the filter width) on quality and efficiency of inpainting automation.
As part of our investigation of 2D NSV for inpainting, we examine the differences between the 2D NSV model and the 2D NSE for inpainting.
We look at semi-implicit forward-time upwind method for both NSV and NSE and compare their stability and efficiency.
We will refer to these numerical methods by their corresponding mathematical model: NSV and NSE numerical model.
Our numerical results show the NSV model, in comparison to NSE, yields a more stable solution to the inpainting process.
That is, the NSV can be computed with larger time-steps, reducing the computational expense in the automated inpainting procedure.
However, we also study the efficiency of the calculation of the model equations as well as the quality of the resulting images.
We hope that these theoretical and numerical studies shed some light on what can be expected from this category of approach when automating the inpainting problem.

This paper is organized as follows.
In Section \ref{prelim} we review the work of Bertalmio {\it et. al.} in~\cite{BBS01} on exploiting an analogy between image intensity and the stream function in 2D incompressible fluid.
In Section \ref{NSV-section} we give some background on the NSV mathematical model, discuss why it is a good candidate sub-grid-scale turbulence model, and argue that it makes for a plausible alternative model for inpainting.
In Section \ref{results-section} we describe our implementation, and then present our numerical results, comparing the resulting images when using the NSE and NSV for inpainting.
The results show the NSV model allows for larger time steps to converge to the steady state solution yielding a more efficient numerical process when automating the inpainting process.
We compare quality of the resulting images using subjective measure (human evaluation) and objected measure (by calculating the peak signal-to-noise ratio (PSNR), also known as peak signal-to-reconstructed measure).
% \begin{comment}
% In order to take into account the extra computational work needed when inverting the Helmholtz equation in the NSV equation, we also compare the explicit NSV to semi-implicit NSE in addition to comparing the two model equations under an explicit discretization scheme.  
% \end{comment}
We also give approximate operation counts needed by the two models to converge to steady state solution for some fixed tolerance.
In Section \ref{sec:theory} we present some supporting theoretical results.
In Section \ref{subsec:uniqueness}, we summarize some existing results on the dependence of uniqueness of the solution to the inpainting problem on the image at the boundary, viscosity and on the size of the inpainting region.      
In Sections \ref{subsec:stab_anal_nse}, \ref{subsec:stab_anal_nsv}, and \ref{subsec:stab_anal_nse_semi}, we establish some new results comparing the NSE and NSV in terms of stability conditions.
We describe the discretization under consideration, and following~\cite{Temam01}, derive sufficient conditions for stability based on energy methods, establishing results that support our numerical results.
In Section~\ref{conclusion} we summarize our results and outline some future directions for this work.  

%%%%%%%%%%%%%%%%%%%%%%%%%%%%%%%%%%%%%%%%%%%%%%%%%%%%%%%%%%%%%%%%%%%%%%%%%%%%%%
\section{Image Inpainting Using Fluid Models}\label{prelim}

 In this section we present a summary of the approach in~\cite{BBS01} for automating the inpainting process motivated by ideas from classical fluid dynamics.
The idea is to propagate isophote lines from the exterior into the region to be inpainted, which we denote as $\Omega$,
using the 2D NSE for fluid dynamics.
Throughout the paper, we denote the digital gray-scale image by $I$, an $m\times n$ matrix with
gray-scale value 0 to 255 at each pixel, and $D$ be the set of points $(x,y)$ where $I$ is defined.
The value 0 represents black and the value 255 represents white.
The values in between are different gray levels between black and white.

\begin{figure}[ht]
\begin{center}
\includegraphics[scale=0.4]{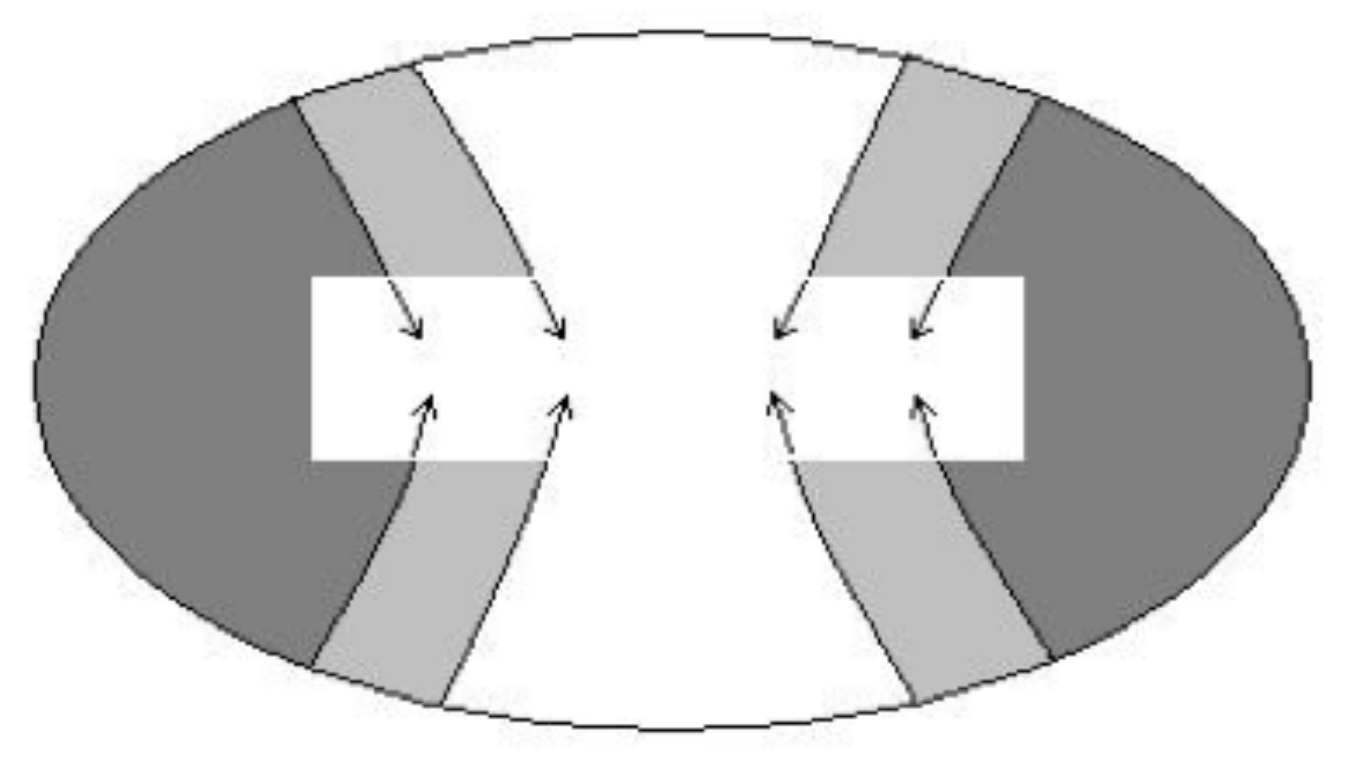}
\caption{\small\em This image shows the direction of the isophotes and where the smoothness should propagate.}
\label{isophotes}
\end{center}
\end{figure}

To solve the inpainting problem, a common technique for restorers (see e.g.,\cite{BSCB00}) is to extend edges from the
boundary of $\Omega$, filling in the intra-region with the correct color gradient.
For example, in Figure~\ref{isophotes}, the direction of the isophotes (level lines of equal color gradients) determines where the smoothness should propagate.
To see how we can automate this technique, we first introduce some basic mathematical concepts.
Mathematically, the direction of {\it isophotes} (level curves of equal gray-levels) can be represented as $\nabla^{\perp}I$, which indicates the
direction of zero change, and the $\it{smoothness}$ of the
image, can be represented by $\Delta I$, where $\Delta$ is the usual Laplacian operator (see~\cite{BBS01} and references therein). 

%%%%%%%%%%%%%%%%%%%%%%%%%%%%%%%%%%%%%%%
\subsection{Inpainting Using the Navier-Stokes Equations}

In~\cite{BSCB00}, Bertalmio, {\em et. al.} proposed that the solution $I^*$ can be approximated with the steady solution to the PDE of the form
\begin{equation}
I_t= \nabla^{\perp}I\cdot \nabla \Delta I +\nu\nabla\cdot(g(|\nabla I|)\nabla I),
\end{equation}
for small $\nu >0$, where the anisotropic diffusion is added to preserve the edges.  We will list below a few diffusivity functions $g$ that are used in the literature.  To be more precise, the main goal is to find the solution $I^*$ such that the level curves of $\Delta I^*$ almost parallel to $\nabla^\perp I^*$, that is,
\begin{equation}
\nabla^{\perp}I^*\cdot \nabla \Delta I^* \simeq 0.
\label{rule}
\end{equation}
Bertalmio {\em et. al.}~\cite{BBS01} exploited an analogy between
the image of intensity function $I$ and the stream
function in a 2D incompressible fluid. To see this, we recall the vorticity-stream formulation of the 2D NSE,
\begin{equation}\label{vte}
\frac{\partial \omega}{\partial t} + v \cdot \nabla \omega = \nu
\Delta \omega.
\end{equation}
Here $v=\nabla^{\perp} \Psi$ is the velocity, where $\Psi$ is the stream function and $\omega=\nabla\times v$ is the vorticity.  If the viscosity $\nu$ is zero, the steady state
solution for the stream-function $\Psi$ in 2D satisfies

\begin{equation}
\label{sss} \nabla^{\perp} \Psi \cdot \nabla \Delta \Psi = 0.
\end{equation}\\
The similarity between (\ref{rule}) and (\ref{sss}) allows one to develop methods
for the inpainting problem using techniques from fluid dynamics.  For the image inpainting problem, instead of simulating (\ref{vte}),
we compute the following numerically:
\begin{equation}\label{vtei}
\frac{\partial \omega}{\partial t} + v \cdot \nabla \omega =
\nu\nabla\cdot (g(|\nabla \omega|)\nabla \omega),
\end{equation}
where $g$ accounts for the anisotropic diffusion (edge preserving
diffusion) to sharpen the image. At each time step, we solve the
Poisson equation:
\begin{equation}
\label{pois}
\Delta I = \omega, \ \ \ I|_{\partial \Omega} = I_0,
\end{equation}
where the boundary values $I_0$ are derived from the values of $I$ in $D - \Omega$.

We recall that other PDEs have been proposed recently (cf.~\cite{BoMarz07,TscDer05}) which have been shown to more efficient and perhaps easier to implement.
Here, our main interest is to explore different sub-grid scale turbulent models applied to image inpainting.
The idea is to test different sub-grid scale turbulence models instead of simulating the NSE (or the equation in (\ref{vtei})) for purposes of finding the solution to the image inpainting problem for reduced computational requirements.
Modifying the function $g$ appropriately in (\ref{vtei}) can also add to the accuracy of preserving the edges and hence give a more accurate image inpainting.
In (\ref{vtei}) we noted that for the image inpainting, the dissipation term in the NSE was modified to accommodate anisotropic diffusion. If $g=1$, we recover the usual NSE with isotropic diffusion.
There has been extensive analysis of various types of anisotropic filtering, see for example~\cite{YXTK96}.
In our numerical simulation of equation (\ref{vtei}) we have used several different diffusivity functions, namely,
\begin{equation*}
g(s) = \left[1 + \left(\dfrac{s}{k^2}\right)^2\right]^{-1},
\quad \quad
g(s) = \exp\left(-\dfrac{s}{k^2}\right),
\quad \quad
g(s) = \left[1 + \left(\dfrac{s}{k^2}\right)\right]^{-1},
\end{equation*}
where $k$ is a predefined diffusion parameter.  For the particular test cases performed, we observed no significant difference in the numerical solution to the image inpainting problem.

\begin{figure*}[htpb]
\begin{center}
\label{imageinpainting}
\subfigure[]{
\includegraphics[scale=0.48]{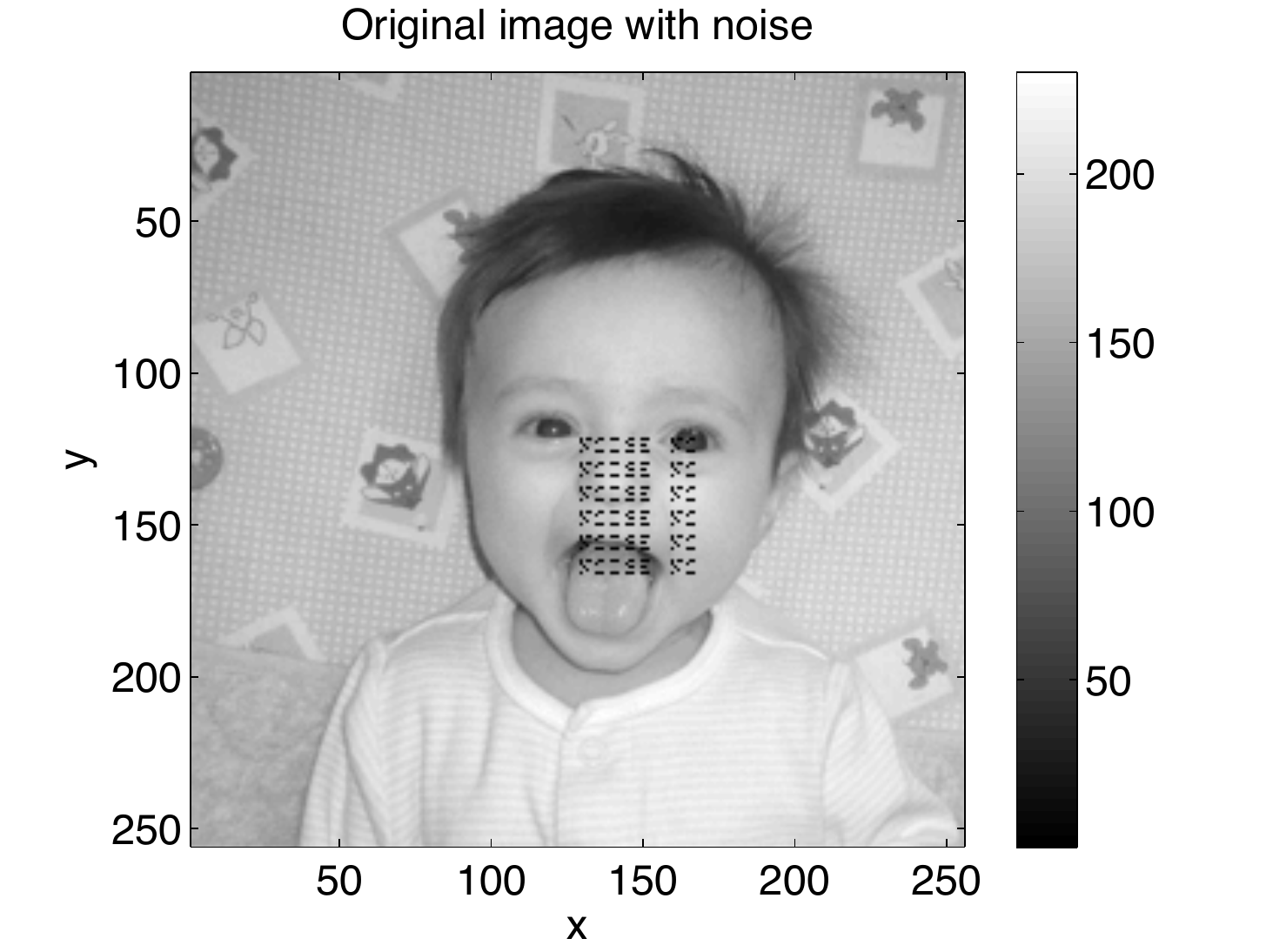}
}
\subfigure[]{
\includegraphics[scale=0.48]{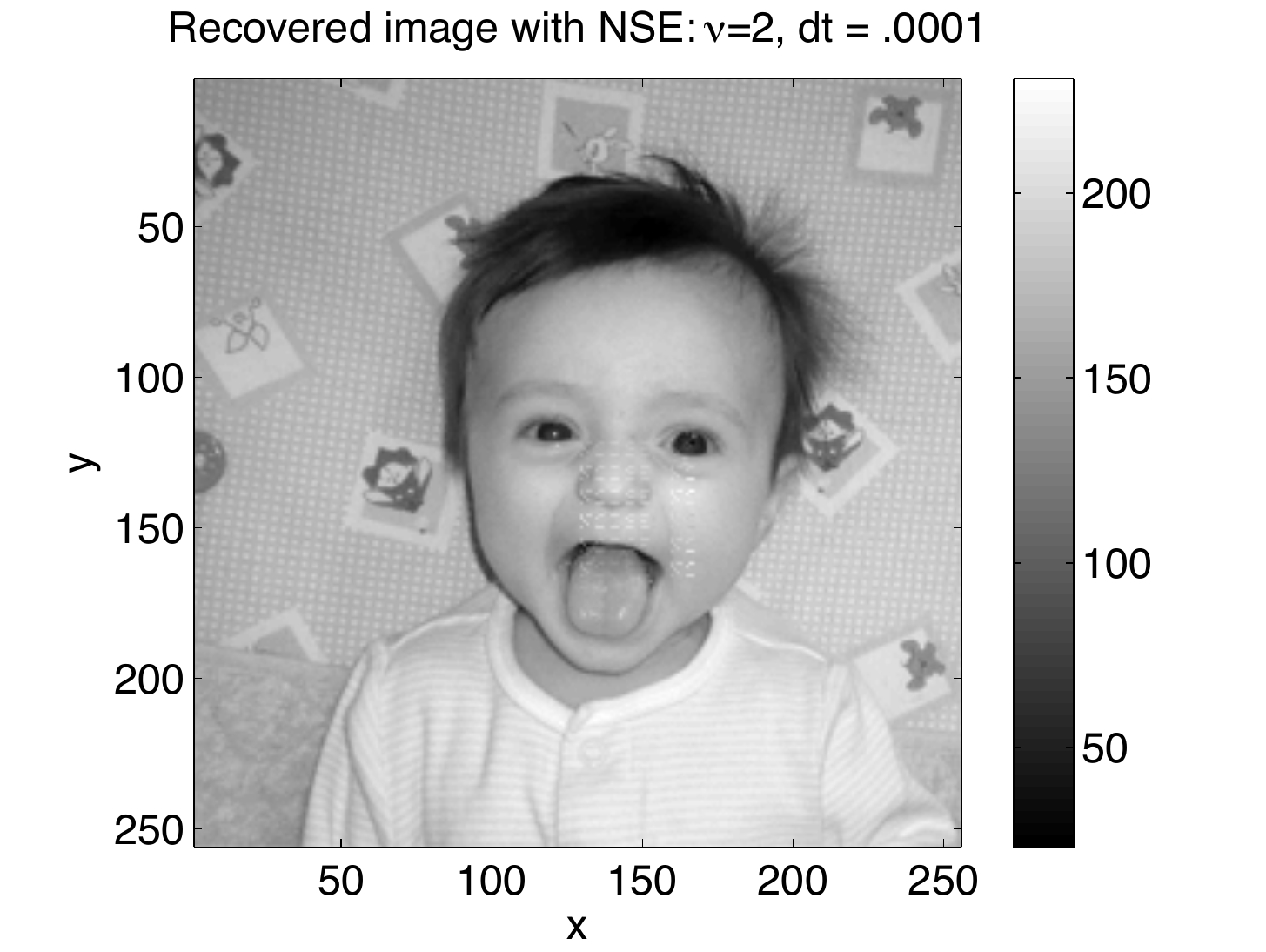}
}
\caption{\small\em This is an example of image inpainting resulting from the steady state solution of the Navier-Stokes equations (with anisotropic diffusion to preserve edges) with viscosity $\nu=2$.  The time-step $dt$ was set to .0001 to guarantee convergence of the numerical solution.}
\end{center}
\end{figure*}
Just as an illustration of the idea, we give a simple numerical experiment. In Figure~\ref{imageinpainting}a, we took the photo of Alexander Ebrahimi  with a noise mask as a test case.  We run the Navier-Stokes equations (NSE) with anisotropic diffusion until its steady state~\cite{BBS01}, and recovered the picture in Figure~\ref{imageinpainting}b.  The iterative inpainting process written in MATLAB took 24 iterations with a tolerance of 0.0001, in about 1.84 seconds on a laptop computer with a Intel(R) 1.60 GHz CPU.  The grid for the inpainting region is $49\times 59$ pixels.

% \begin{comment} 
% Note that if we look closely at the inpainted image in Figure~\ref{imageinpainting}b, we can see slight traces of where the noise mask was.  However, if a typical viewer is not looking for such a flaw then the photo in Figure~\ref{imageinpainting}b is acceptable.
% \end{comment}

%%%%%%%%%%%%%%%%%%%%%%%%%%%%%%%%%%%%%%%%%%%%%%%%%%%%%%%%%%%%%%%%%%%%%%%%%%%%%%
\subsection{Alpha Models and the Navier-Stokes-Voight (NSV) Equations}\label{NSV-section}

The NSV equations \pref{nsv} were suggested as a regularization model for the 3D NSE, where the length-scale $\alpha$ is considered as the regularizing parameter.  The system (\ref{nsv}) was first introduced in 1973 by Oskolkov (see~\cite{Osk73,Osk80}) as a model of a motion of linear, viscoelastic fluid.  In that setting, $\alpha$ is thought of as a length-scale parameter characterizing the elasticity of the fluid.  As noted in~\cite{KLT07,KT07}, the addition of the term $-\alpha^2 \Delta u_t$ regularizes the 3D NSE which makes it globally well-posed~\cite{CLT06,Osk73} and  it changes the parabolic character of the  original Navier-Stokes equations where in this case one does not observe an immediate smoothing of the solutions as usually expected in parabolic PDEs.  Instead, the equations behave like a damped hyperbolic system.  Despite the damped hyperbolic behavior, it was shown in~\cite{KLT07} that the solutions to the 3D NSV equations lying on the global attractor posses a dissipation range.  This property supports the claim that indeed, the NSV equations can be used as a sub-grid scale turbulence model.  This type of inviscid regularization (that is, a regularization technique without introducing extra viscous or hyperviscous terms) has been used for 2D surface quasi-geostrophic (SQG) model (\cite{KhTi07}), see also~\cite{BLT07} for the Birkhoff-Rott equation, induced by the 2D Euler-$\alpha$ equations for vortex sheet dynamics and~\cite{LiTi07} for the 3D Magnetohydrodynamics (MHD) equations.

There is an interesting literature behind the rediscovery of the NSV equation as a turbulence model.
It can be traced back from the early study of 3D Navier-Stokes-$\alpha$ (NS-$\alpha$) turbulence model in 1998 (also known as the viscous Camassa-Holm equations (VCHE) and Lagrangian averaged Navier-Stokes-$\alpha$ (LANS-$\alpha$) model, which can be written as~\cite{FHT01,FHT02,HMR98,MM03}
\begin{equation}\label{nsa}
\aligned
\partial_t v - \nu\Delta v - u\times\nabla\times v &= -\nabla p + f,\\
\nabla \cdot u =  \nabla \cdot v &= 0,\\
v &= (I-\alpha^2\Delta)u,
\endaligned
\end{equation}
with $u(x,0) = u^{in}(x)$.
The parameter $\alpha$ can be viewed as the length scale associated with filter width smoothing $v$ to obtain $u$, with filter associated with the Green's function (Bessel potential) of the Helmholtz operator $(I-\alpha^2\Delta)$.
The system is supplemented with periodic boundary conditions in a box $[0,L]^3$.

The inviscid and unforced version of the 3D NS-$\alpha$ was introduced in~\cite{HMR98} based on the Hamilton variational principle subject to the incompressibility constraint $\mbox{div}\; v =0$.
By adding the viscous term $-\nu\Delta v$ and the forcing $f$ in an {\it ad hoc} fashion, the authors in~\cite{CFHOT98,CFHOT99a,CFHOT99b} and~\cite{FHT02} obtain the NS-$\alpha$ system which they named, at the time, the viscous Camassa-Holm equations (VCHE), also known as the Lagrangian averaged Navier-Stokes-$\alpha$ model (LANS-$\alpha$).
In~\cite{CFHOT98,CFHOT99a,CFHOT99b} it was found that steady state solutions of the 3D NS-$\alpha$ model compared well with averaged experimental data from turbulent flows in channels and pipes for large Reynolds numbers.
This led researchers~\cite{CFHOT98,CFHOT99a,CFHOT99b} to suggest that the NS-$\alpha$ model be used as a closure model for the Reynolds averaged equations.
Since then, an entire family of `$\alpha$'- models has been found which provide similar successful comparison with empirical data -- among these are the Clark-$\alpha$ model~\cite{CHT05}, the Leray-$\alpha$ model~\cite{CHOT05}, the modified Leray-$\alpha$ model~\cite{ILT06} and the simplified Bardina model~\cite{CLT06, LL06} (see also~\cite{OlTi07} for a family of similar models).
We focus our attention on the simplified Bardina model:
\begin{equation}\label{Bardina}
\aligned
\partial_t v  -\nu \Delta v+ (u\cdot\nabla) u &= -\nabla p + f,\\
\nabla \cdot u =  \nabla \cdot v &= 0,\\
v &= u-\alpha^2\Delta u,
\endaligned
\end{equation}
with initial condition $u(x,0) = u^{in}(x)$.
The equation above was introduced and studied in~\cite{LL06} supplemented with periodic boundary conditions.  Notice that consistent with all the other alpha models, the above system is the Navier-Stokes system of equations when
$\alpha = 0$, i.e. $u = v$.  In~\cite{CLT06}, the viscous and inviscid simplified Bardina models were shown to be globally well-posed.  It was also shown that the viscous simplified Bardina model has a finite dimensional global attractor.
In~\cite{CLT06} it was observed that the inviscid simplified Bardina model, is equivalent to the following modification of the 3D Euler equations
\begin{equation}
\aligned
-\alpha^2\Delta\dfrac{\partial u}{\partial t} + \dfrac{\partial u}{\partial t} + (u\cdot \nabla)u + \nabla p &= f, \\
\nabla\cdot u &= 0,
\endaligned
\end{equation}
with initial condition $u(x,0)=u^{in}$.
In particular, it is equal to the Euler equations when $\alpha = 0$.
Following this observation, Cao, {\it et. al.}~\cite{CLT06} proposed the inviscid simplified Bardina model as regularization of the 3D Euler equations that could be used for simulations of 3D inviscid flows.
Inspired by the above model, Cao, {\it et. al} then
proposed the following regularization of the 3D Navier-Stokes
equations
\begin{equation}\label{new-mod}
\aligned
-\alpha^2\Delta\dfrac{\partial u}{\partial t} +
\dfrac{\partial u}{\partial t} -\nu\Delta u+
(u\cdot \nabla)u + \nabla p &= f, \\
\nabla\cdot u &= 0,
\endaligned
\end{equation}
with initial condition $u(x,0)=u^{in}$, and
subject to either periodic boundary condition or the no-slip
Dirichlet boundary condition $u|_{\partial\Omega} = 0$.  In the
presence of physical boundaries the above regularization
(\ref{new-mod}) of the Navier-Stokes equations, is different in
nature from any of the other alpha regularization models,
because it does not require any additional boundary conditions.

This newest addition to the family of `$\alpha$' models was later on discovered to have already existed in the literature known as the Navier-Stokes Voight equations but as a model for viscoelastic fluids~\cite{Osk73,Osk80}. As mentioned earlier, the $\alpha$ in that setting represents the elasticity of the fluid.

%%%%%%%%%%%%%%%%%%%%%%%%%%%%%%%%%%%%%%%
\subsection{The Alpha Model as Sub-Grid-Scale Turbulence Model}

In addition to the remarkable match of explicit analytical steady state solutions of the alpha models to the experimental data in channels and pipes, the validity of the first alpha model, namely the NS-$\alpha$ model, as a sub-grid scale turbulence model was also tested numerically in~\cite{CHMZ99,MM03}.
In the numerical simulation of the 3D NS-$\alpha$ model, the authors of~\cite{CHMZ99,GH99,GH03,MM03} showed that the large scale (to be more specific, those scales of motion bigger than the length scale $\alpha$) features of a turbulent flow is captured.
For scales of motion smaller than the length scale $\alpha$, the energy spectra decays faster in comparison to that of NSE.
This numerical observation has been justified analytically in~\cite{FHT01}.
In direct numerical simulation, the fast decay of the energy spectra for scales of motion smaller than the supplied filter length represents reduced grid requirements in simulating a flow.
The numerical study of~\cite{CHMZ99} gives the same results, which also hold for the Leray-$\alpha$ model in~\cite{CHOT05,GH99}.

In~\cite{LKTT07} it was observed that the scaling exponent of the energy spectrum of the 2D Navier-Stokes-$\alpha$ model (NS-$\alpha$), for wave numbers $k$ such that $k\alpha\gg 1$, is $k^{-7}$.
{\it A posteriori}, it was observed that this scaling corresponding to that predicted by assuming that the dynamics for $k\alpha\gg 1$ was governed by the characteristic time scale for flux of the conserved enstrophy.
This observation led researchers~\cite{LKTT07} to speculate that (in general) the unknown scaling exponent for any $\alpha$-model may be predicted by the dynamical time scales for the dominant conserved quantity for that model in the regime $k\alpha\gg 1$.
This speculation is verified to be correct in~\cite{LKT08}, where numerical simulations of the 2D Leray-$\alpha$ model were presented.

%%%%%%%%%%%%%%%%%%%%%%%%%%%%%%%%%%%%%%%
\subsection{Navier-Stokes-Voight for Image Inpainting}

The main difference between NSV and other alpha models is that it needs no additional boundary conditions in the presence of physical boundaries. This makes NSV a natural alternative to NSE for inpainting.
In order to solve the inpainting problem we need to compute the solution to (\ref{sss}). We approximate it with the steady state solutions of (\ref{vtei}) with viscosity $\nu>0$.  As noted in~\cite{BBS01} the presence of viscosity is needed to have a unique steady-state solution and its stability depends on how big or small is the viscosity.  It is well-known that for very high Reynolds number flows (ie very small viscosity $\nu$), direct numerical simulation of the NSE, requires computational resources which cannot be accessed easily.  For purposes of direct numerical simulations, the NSV equations were proposed as a sub-grid scale turbulence model. Here we would like to investigate the effects of this regularized equations to image inpainting.  Notice that the steady state solution to the NSV equations in (\ref{nsv}) matches the steady state solution to the NSE.  We expect that by using the NSV, the convergence to the steady state solution will be at a much faster rate reducing the computational cost.

%% file: numerics.tex
%%%%%%%%%%%%%%%%%%%%%%%%%%%%%%%%%%%%%%%%%%%%%%%%%%%%%%%%%%%%%%%%%%%%%%%%%%%%%%
\section{Numerical Comparison NSE and NSV for Image Inpainting}\label{results-section}

We simulate the inviscid 2D NSV  with anisotropic diffusion (artificial viscosity)
\begin{equation}\label{vtei-alpha}
-\aa^2\frac{\partial}{\partial t}\Delta\omega+\frac{\partial \omega}{\partial t} + u \cdot \nabla \omega =
\nu\nabla\cdot (g(|\nabla \omega|)\nabla \omega),
\end{equation}
where, $\omega = \nabla \times u$. Notice that when $\alpha =0$, we recover (\ref{vtei}).  Using a forward time up-wind finite-difference scheme \cite{ATXX}, we solve for $\omega_{i,j}^{(n+1)}$ in the discretized equation:
\begin{equation}
\begin{split}
& (1-\alpha^2 \Delta) \omega_{i,j}^{n+1}
- \nu dt\;(\partial_x(g\omega_x^{n+1})+\partial_y(g\omega_y^{n+1}))
= (1-\alpha^2 \Delta) \omega_{i,j}^n \\
& + dt\left[ -|u_1^n| \mbox{ sgn}(u_1^n)
\dfrac{\omega_{i+ \mbox{ sgn}(u_1^n),j}^n-\omega_{i,j}^n}{dx} 
- |u_2^n| \mbox{ sgn}(u_2^n) \dfrac{\omega_{i,j+\mbox{ sgn}(u_2^n)}^n-\omega_{i,j}^n}{dy} \right].
\label{explicit}
\end{split}
\end{equation}
% \begin{comment}
% Note that we use the Jacobi iteration method in finding $\omega_{i,j}^{(n+1)}$ for the above system.  We also use the Jacobi iteration method when solving the Poisson equation with Dirichlet boundary condition to get $I^{(n+1)}_{i,j}$.
% \end{comment}
% and then solve for $I^n$ in the equation $\Delta I^n = \omega^n$, using the standard centered discretization scheme.
% 
% For the finite difference scheme above, boundary conditions for both $I^n$ and $\omega^n$ are required. We use the Dirichlet conditions with the values of $I$ on $\partial \Omega$. In computing the values of $\omega$, $u_1$, and $u_2$ we use second order finite difference methods, such as centered difference or one sided three point formula \cite{ATXX}.  We use the same numerical scheme in the case for the 2D NSE ($\alpha = 0$).  In \cite{BBS01} they evolve the vorticity transport equation using forward Euler time stepping, with centered differences in space for the diffusion and a min-mod method for the convection term.  The diffusion is also anisotropic.  

Our numerical experiments can be summarized into three categories: test for stability, test for efficiency and test for quality. In section \ref{stability}, we present some test cases showing that for a given time stepping size, the gray-level blows up when the inpainting is done using NSE while the inpainting using NSV produced a stable solution.  We give theoretical arguments in the appendix which in some sense support these numerical results.  It is worthwhile to mention that using an explicit numerical scheme for both the 2D NSE and 2D NSV equations, one can see a nice advantage of using NSV instead of the NSE in automating the inpainting problem in terms of CFL condition when we for simplicity ignore the nonlinear term in the governing equations.  In particular, the usual stiffness of the NSE by discretizing the linear part
explicitly is no longer present for NSV.  In the Appendix, we present rigorous argument for the stability condition when the nonlinear term is taken into consideration.
% \begin{comment}
% In section \ref{vary-size-e} we vary the size of the inpainting region and compare the number of iterations needed in order to converge to the steady state solution for the two model equations NSE and NSV.
% In section \ref{vary-size-si} we take into consideration the extra solve needed when implementing the NSV model equations.   We vary the size of the inpainting region and compare the number of iterations needed in order to converge to the steady state solution for the two model equations NSE (using semi-implicit scheme) and NSV (using explicit scheme).
% \end{comment}
In section \ref{quality-efficiency} we present some numerical studies using NSE and NSV for different values of $\alpha$.  We look at the effect of the parameter $\alpha$ on the quality of the resulting images and estimate the number of floating point operations (FLOPS) needed to converge to steady state.  We perform this numerical experiment for two different time-steps.  The quality of the resulting image is measured using PSNR which will be defined in the next section.

%%%%%%%%%%%%%%%%%%%%%%%%%%%%%%%%%%%%%%%
\subsection{Stability Behavior of NSE and NSV for Inpainting}\label{stability}

\begin{table}[ht]
  \caption{\small\em Comparison between NSE($\alpha = 0$) and NSV.   $\nu$ = viscosity coefficient, $dt$- size of time step, $\Omega$ (in pixels) = size of the inpainting region.}

\begin{tabular}{c|lclc|l}\hline
   Figure  & $\alpha$ & $\nu$& $dt$ & $\Omega$ & Recovered image \\
\hline\hline
3a &   &      & & $10\times 10$ & Original image\\
3b &0 & 2 & 0.1   &               & Does not converge\\
3c & 0.5&  & 0.1      &               & Converges\\
3d & 0&  & 0.001      &               & Converges\\
\hline
4a &   &   & & $64\times 12$ & Original image\\
4b &0 &2 & 0.1   &               & Does not converge\\
4c &0.9 &  & 0.1   &               & Converges\\
4d &0 &  & 0.01   &               & Converges\\
%\hline
%5a &   &   & & $15\times 9$ & Original image\\
%5b & 0&2  & 0.1   &               & Does not converge\\
%5c & 0.5&  & 0.1  &               & Converges\\
%5d & 0&  & 0.01  &                & Converges\\
%\hline
%6a &  &   & & $49\times 59$ & Original image\\
%6b &0&1   & 0.01  &               & Does not converge\\
%6c &0.5 &  & 0.01   &               & Converges\\
%6d & 0&  & 0.001   &               & Converges\\
\hline
\end{tabular}
\label{4test}
\end{table}
In this section we will present two test cases comparing the NSE and NSV both with anisotropic diffusion.  The results presented here suggests that for some fixed specified time-stepping size $dt$,  the NSV produced a stable solution in the image inpainting iteration process while the NSE with anisotropic diffusion provided an unstable solution (gray-level blows up).  A quick summary of results is provided in Table \ref{4test}.

For each test cases, we fix the viscosity $\nu$, tolerance tol, and the inpainting region $\Omega$.  In Figure~\ref{test1}a, the inpainting region was computed in a grid (in pixels) of $10 \times 10$.  For $dt=0.1$ and $\nu=2$, the NSE did not converge to the steady state solution due to large time-stepping, as shown in Figure~\ref{test1}b.  On the other hand, for the same $dt =0.1$ and viscosity $\nu=2$,  the steady state solution to NSV with $\alpha=0.5$ gave a reasonable image inpainting as shown in Figure~\ref{test1}c.  In Figure~\ref{test1}d, we show the converged steady state solution of NSE for a smaller $dt=0.001$. Using a subjective measure, Figures~\ref{test1}c and \ref{test1}d are comparable in quality but Figure~\ref{test1}d required a larger time-step to converge to the steady state solution.

\begin{figure}[ht]
\begin{center}
\subfigure[]{
\includegraphics[scale=0.40]{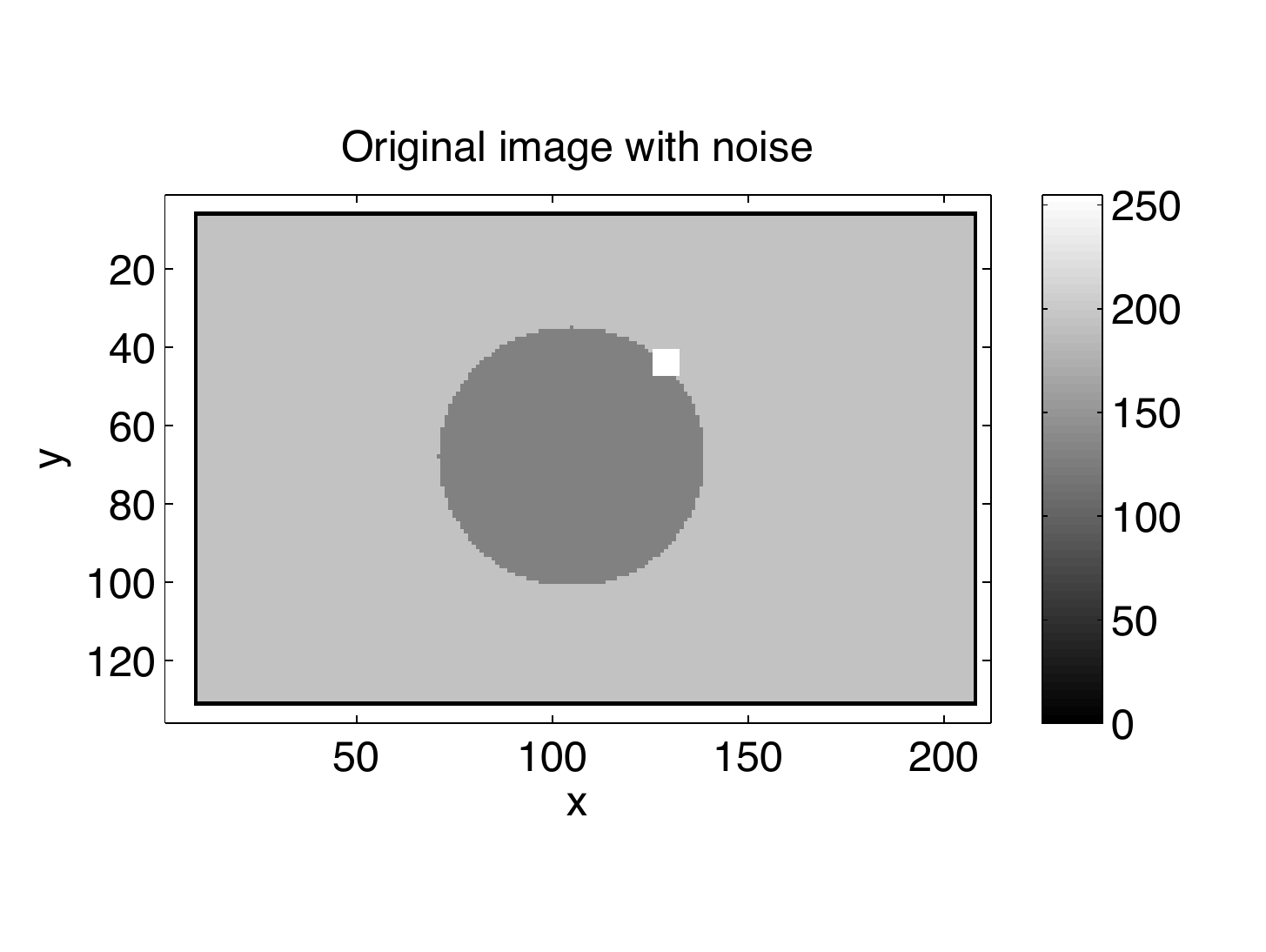}
}
\subfigure[]{
\includegraphics[scale=0.40]{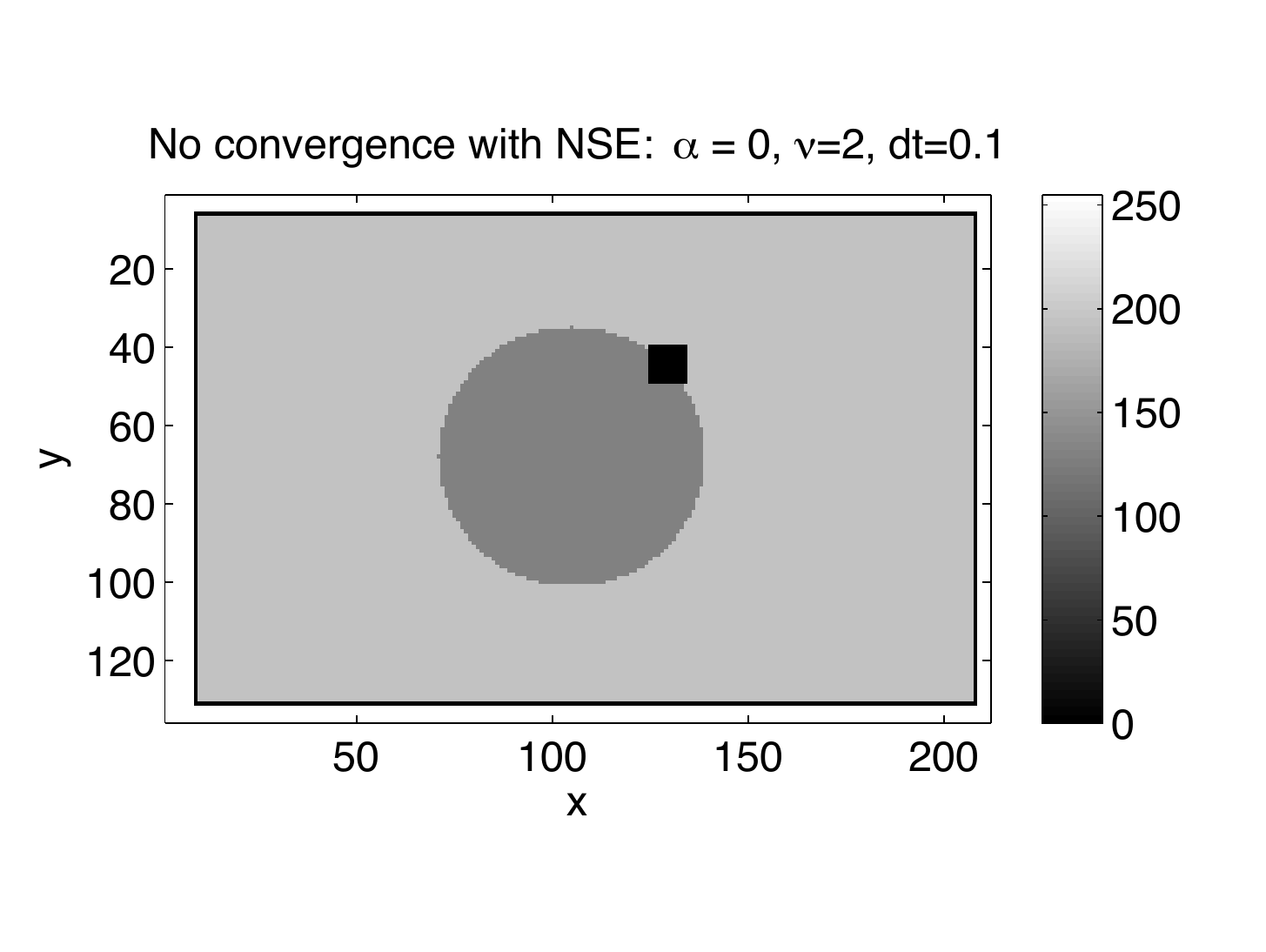}
}
\subfigure[]{
\includegraphics[scale=0.40]{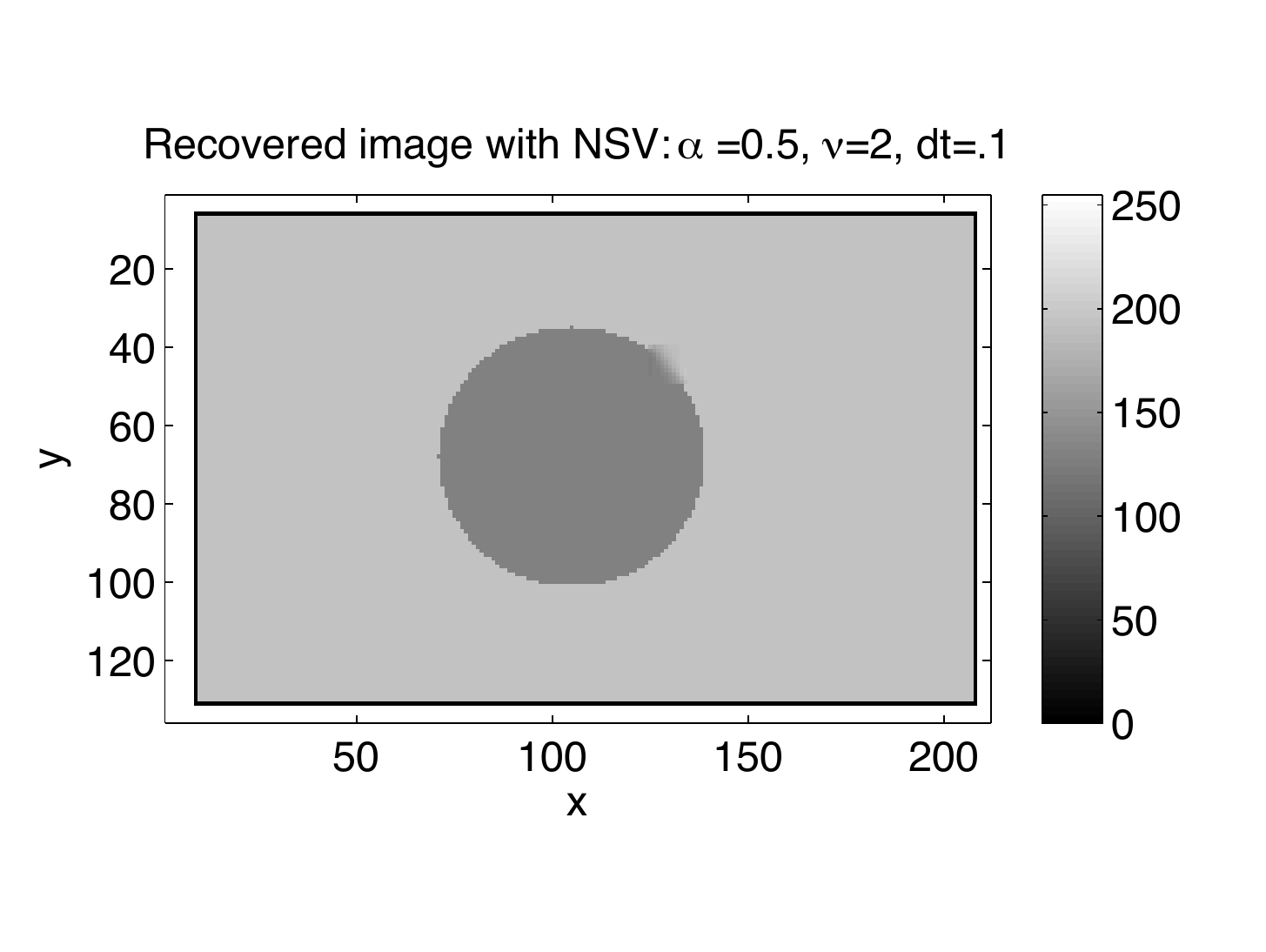}
}
\subfigure[]{
\includegraphics[scale=0.40]{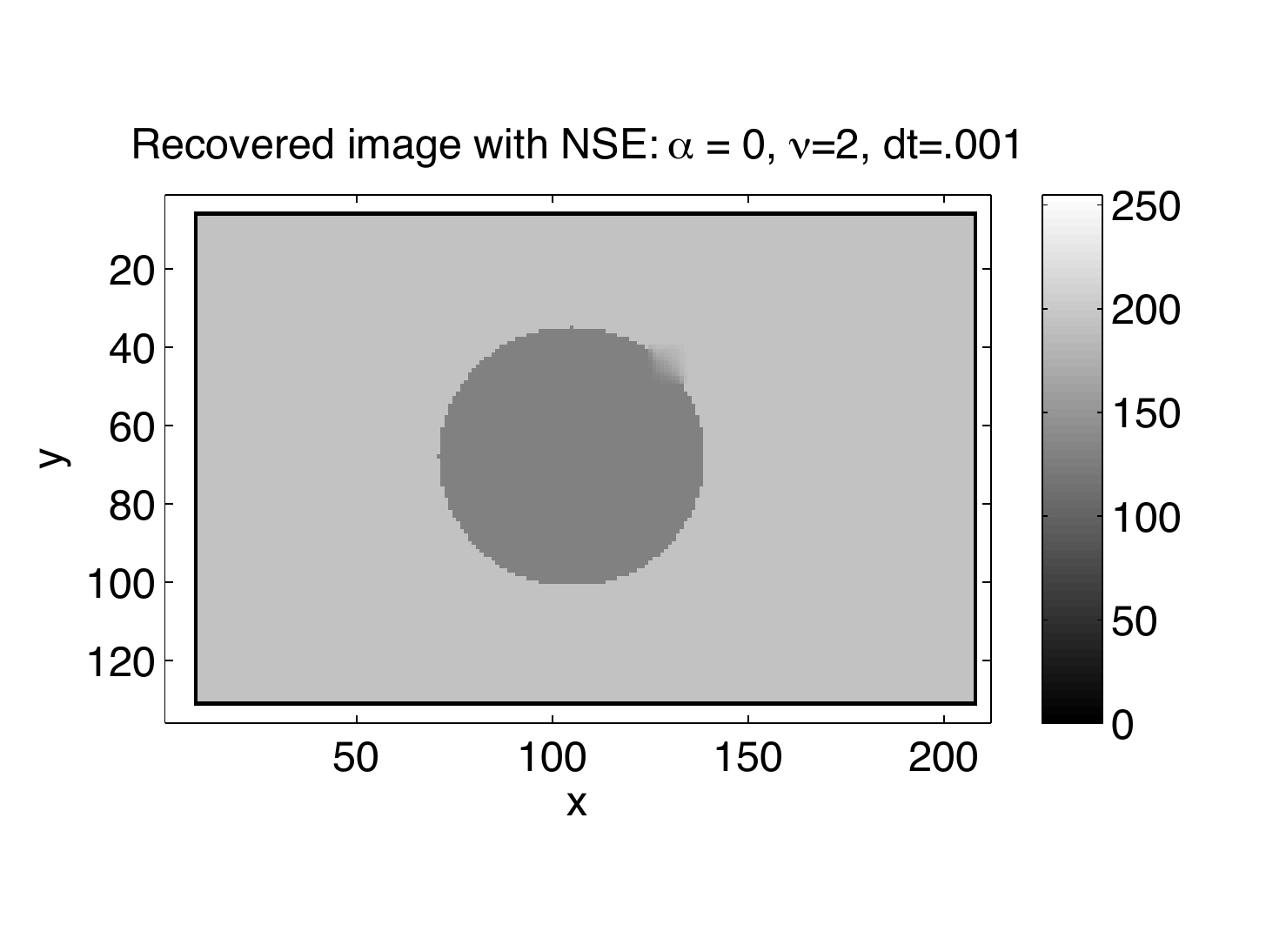}
}
\caption{\small\em A comparison between the NSE and NSV.  Figure a:  Original image $I$ with the inpainting region.  Figure~\ref{test1}b: For step size $dt=0.1$ the NSE with anisotropic diffusion does not converge to its steady state solution. Figure~\ref{test1}c:  On the other hand the NSV with anisotropic diffusion does converge for $\alpha=0.5$ and $dt=0.1$.  Figure d:  For a smaller step size $dt=.001$, the NSE with anisotropic diffusion does converge.}
\label{test1}
\end{center}
\end{figure}[ht]
In Figure~\ref{test2}a, we choose a slightly different size of the inpainting region computed in a grid (in pixels) of $64 \times 12$.  For $dt=0.1$ and $\nu=2$, the NSE did not converge to the steady state solution due to large time-stepping, as shown in Figure~\ref{test2}b.  Similar to the first example, for the same $dt =0.1$ and viscosity $\nu=2$,  the steady state solution to the NSV model, with $\alpha=0.9$, gave a reasonable result to the image inpainting problem as shown in Figure~\ref{test2}c.  In Figure~\ref{test2}d, we show the converged steady state solution of NSE for a smaller $dt=0.01$. Figures~\ref{test2}c and \ref{test2}d are comparable in quality but Figure~\ref{test2}d required a larger time-step to converge to the steady state solution.
\begin{figure}[ht]
\begin{center}
\subfigure[]{
\includegraphics[scale=0.40]{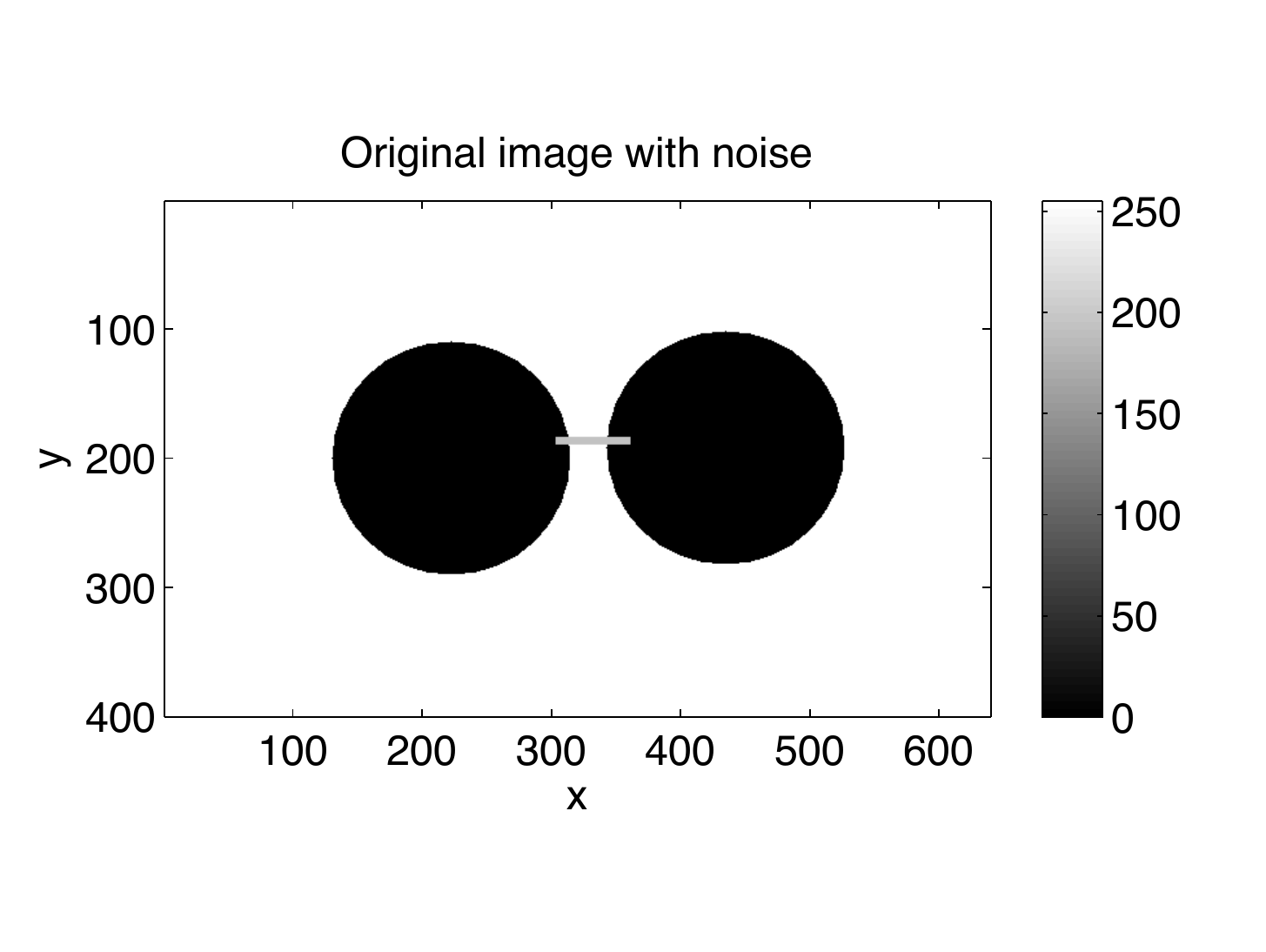}
}
\subfigure[]{
\includegraphics[scale=0.40]{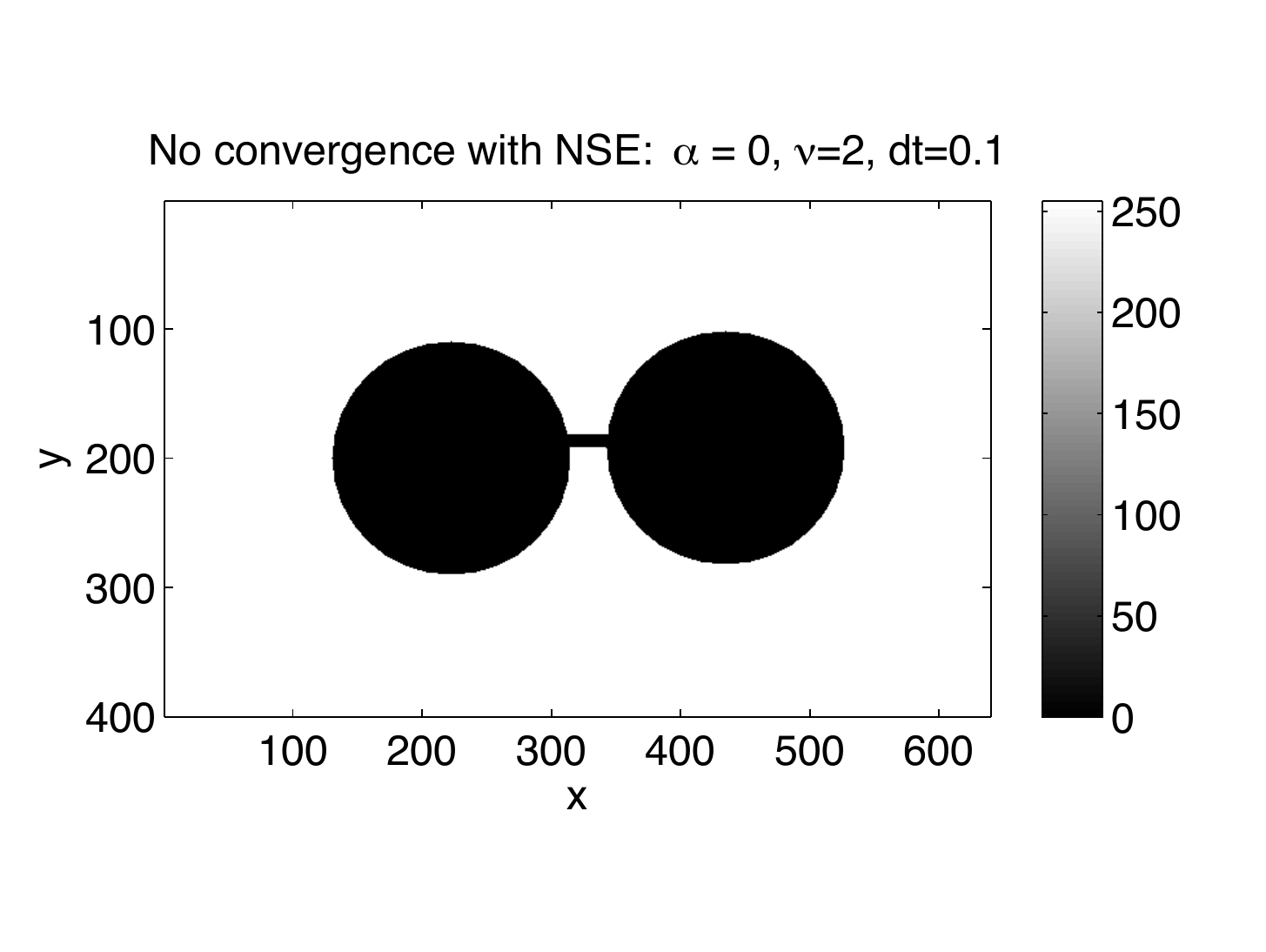}
}
\subfigure[]{
\includegraphics[scale=0.40]{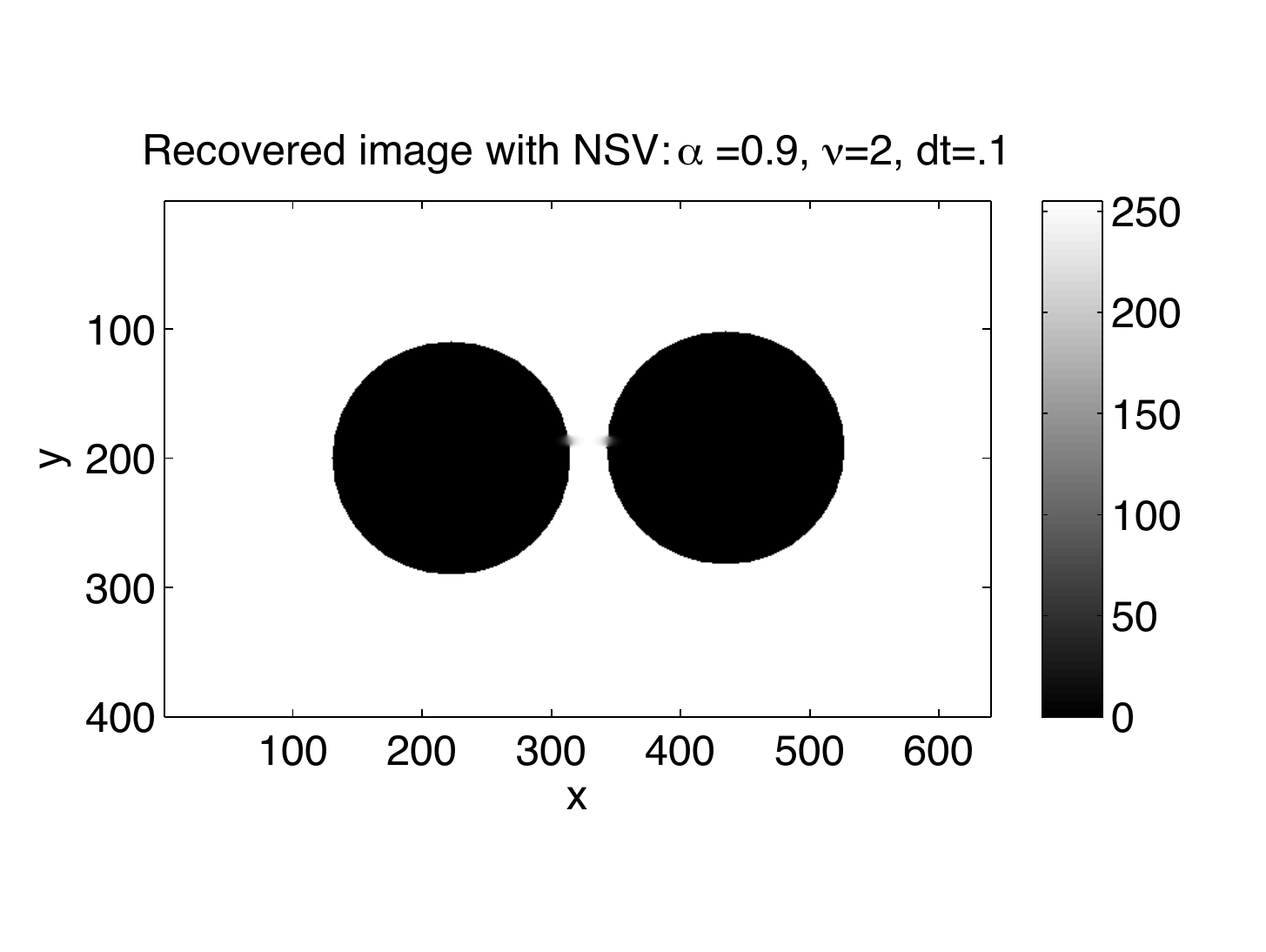}
}
\subfigure[]{
\includegraphics[scale=0.40]{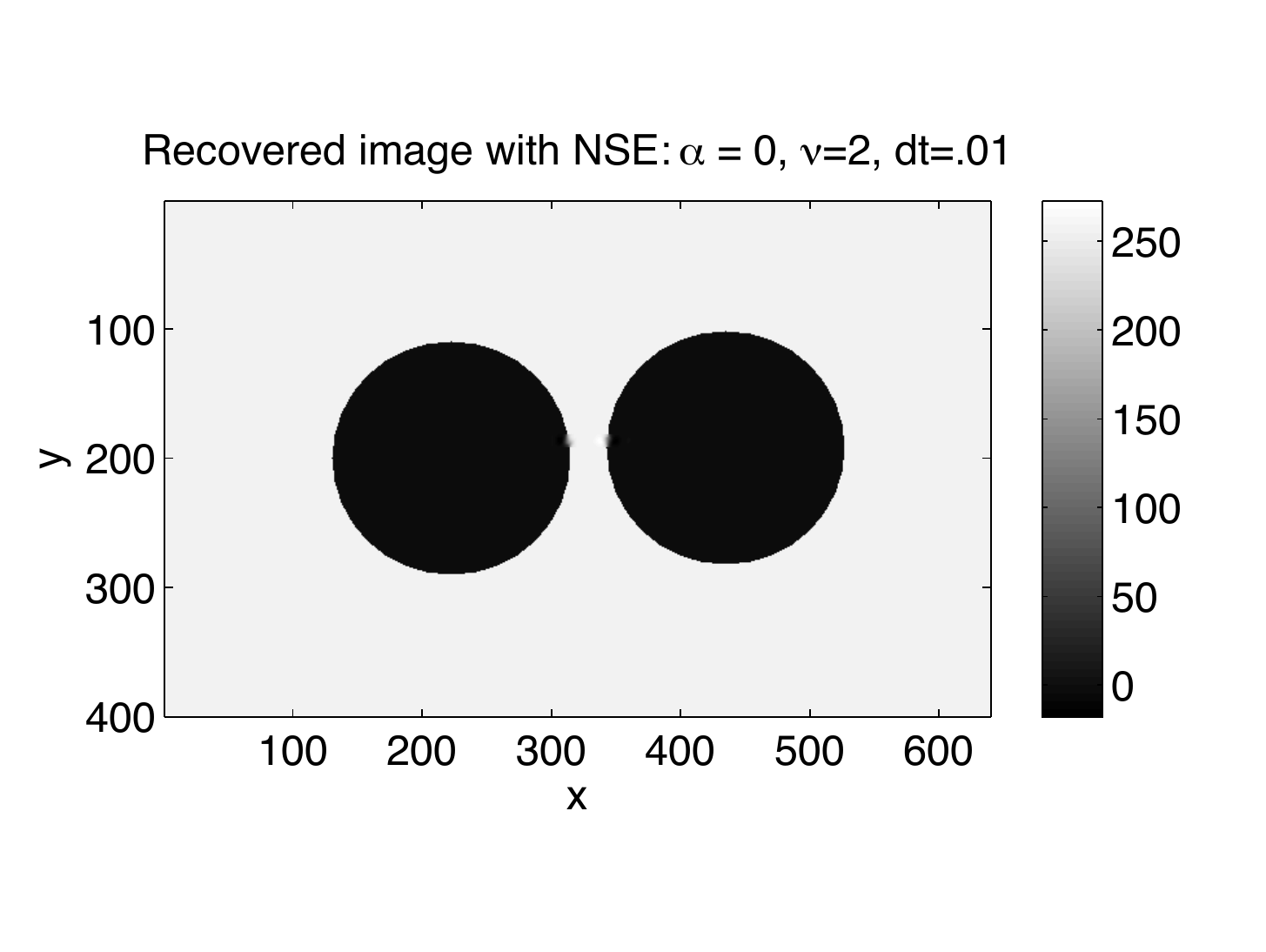}
}
\caption{\small\em A comparison between the NSE and NSV.  Figure a:  Original image $I$ with the inpainting region.  Figure\ref{test2}b:~For step size $dt=0.1$ the NSE with anisotropic diffusion does not converge to its steady state solution. Figure~\ref{test2}c:  On the other hand the NSV with anisotropic diffusion does converge for $\alpha=0.9$ and $dt=0.1$.  Figure~\ref{test2}d:  For a smaller step size $dt=.001$, the NSE with anisotropic diffusion does converge.}
\label{test2}
\end{center}
\end{figure}

\subsection{Comparing the Efficiency of NSE and NSV for Inpainting}\label{quality-efficiency}

The NSV converges to its steady state solution using a larger time-step or 
fewer number of iterations. 
However, it is not clear whether it is more efficient than the NSE because 
of the extra computational cost, for the $\alpha$ term, for each iteration.
In addition, we also need to compare the quality of the resulting images.
We do a FLOP counting for the two model equations until it reached a steady state solution by using the Lightspeed MATLAB toolbox.  We approximated the FLOP by taking only in consideration the cost of matrix inversions.  We use the {\it FLOP\_inv} function which returns the number of FLOP necessary to invert an $n\times n$ matrix. Note that this function returns the minimal number of FLOP for matrix inversion, as though the best possible algorithm was used, no matter which method we are actually using.  This should not be an issue for we use the same matrix inversion technique when running the two model equations.   Given that $\Omega$ is $N\times M$, we take $n=\max(N,M)$.  We then normalize the FLOP count by dividing it by the number of pixels in $\Omega$.  That is, the FLOP count is in units of number of FLOP per pixel.  In addition, we also compute the PSNR of the resulting images to give some measure of comparison for the quality of the resulting images in addition to subjective measure.  Let $P(i,j)$ be the original image that contains $N$ by $M$ pixels and $I(i,j)$ the reconstructed image.  The pixel values range between black (0) and white (255).  The PSNR in decibels (dB) are computed as follows:
\begin{equation*}
PSNR = 20\log_{10}\left(\dfrac{255}{RMSE}\right),
~\textrm{where}~
RMSE= \sqrt{\dfrac{\sum_i\sum_j\left(P(i,j)-I(i,j)\right)^2}{N * M}}.
\end{equation*} 
\begin{table}[ht]
  \caption{\small\em Comparison between NSE($\alpha = 0$) and NSV.   $\nu$ = viscosity coefficient, $dt$- size of time step, FLOP count = FLOP count/(N*M), PSNR = Peak signal-to-noise ratio.}
\begin{tabular}{c|lclrl}\hline
   FLOPS & $\alpha$ & $dt=0.001$ & $dt=0.0001$ &  \\
\hline\hline
 & 0     &  1.2375e4     & 9.0000e3   \\
 & $1/3$ &  6.3753e3     & 5.6253e3  \\
 & $2/3$ &  4.1252e3     & 3.7502e3  \\
 & $1$   &  3.7502e3     & 3.3751e3  \\
 & $4/3$ &  3.3751e3     & 3.7501e3  \\
\hline
\end{tabular}
\label{table2}
\end{table}
Comparison between two PSNR values for different reconstructed image gives one measure of quality.
Reconstructed images with higher PSNR are considered better.
(However, as frequently mentioned in the literature, PSNR may not equate with human perception; we will view PSNR as a reasonable metric here.)
% For example, one can obtain two images with the same PSNR which are are not of the same subjective quality.
% However, in most cases the PSNR is a good way to measure quality.      
In any case, we now describe of our results based on observations in Table~\ref{table2} and Figure ~\ref{psnr}.  In this numerical experiment, we fix the value of $\nu=2$.
\begin{figure}[ht]
\begin{center}
\subfigure[]{
\includegraphics[scale=0.42]{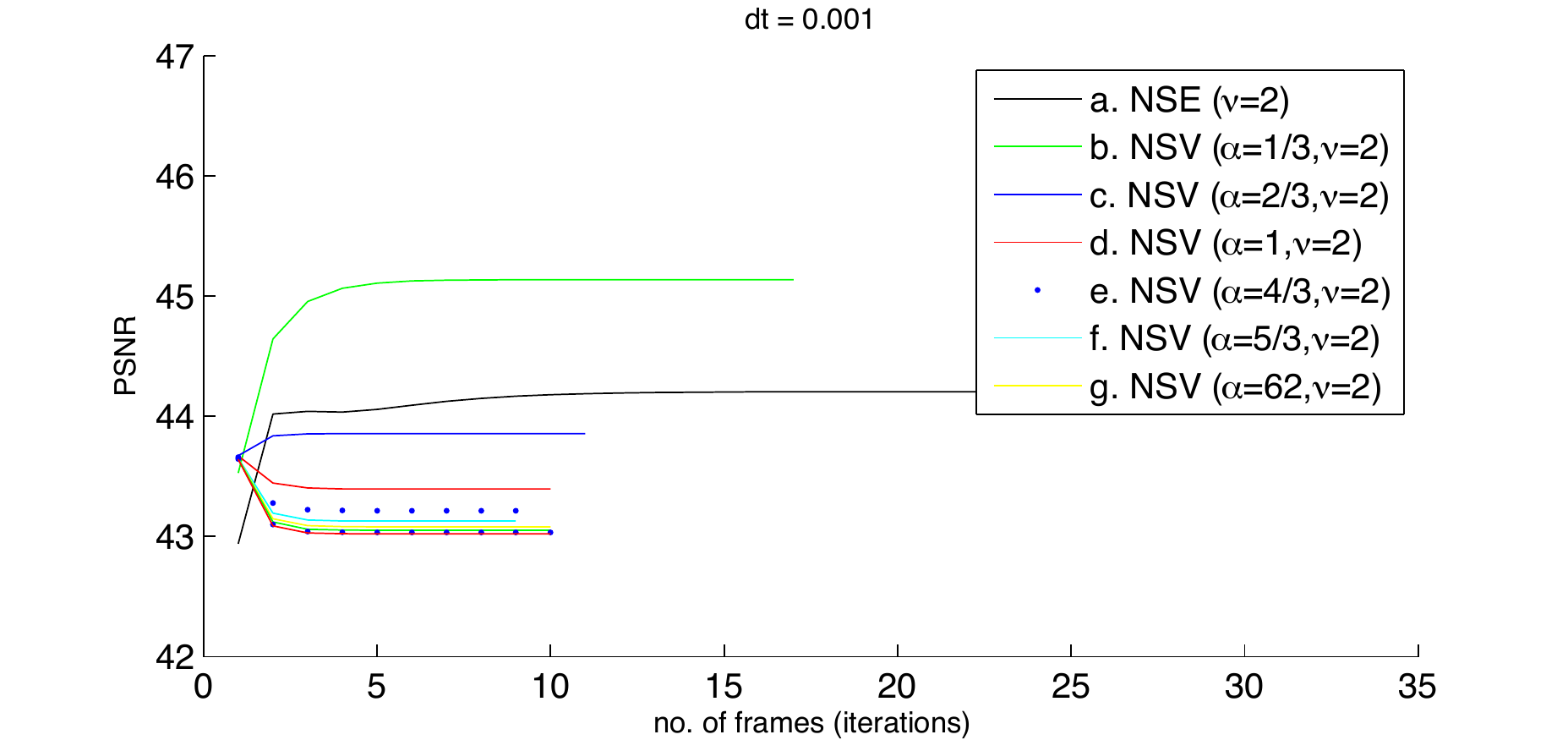}
}
\subfigure[]{
\includegraphics[scale=0.42]{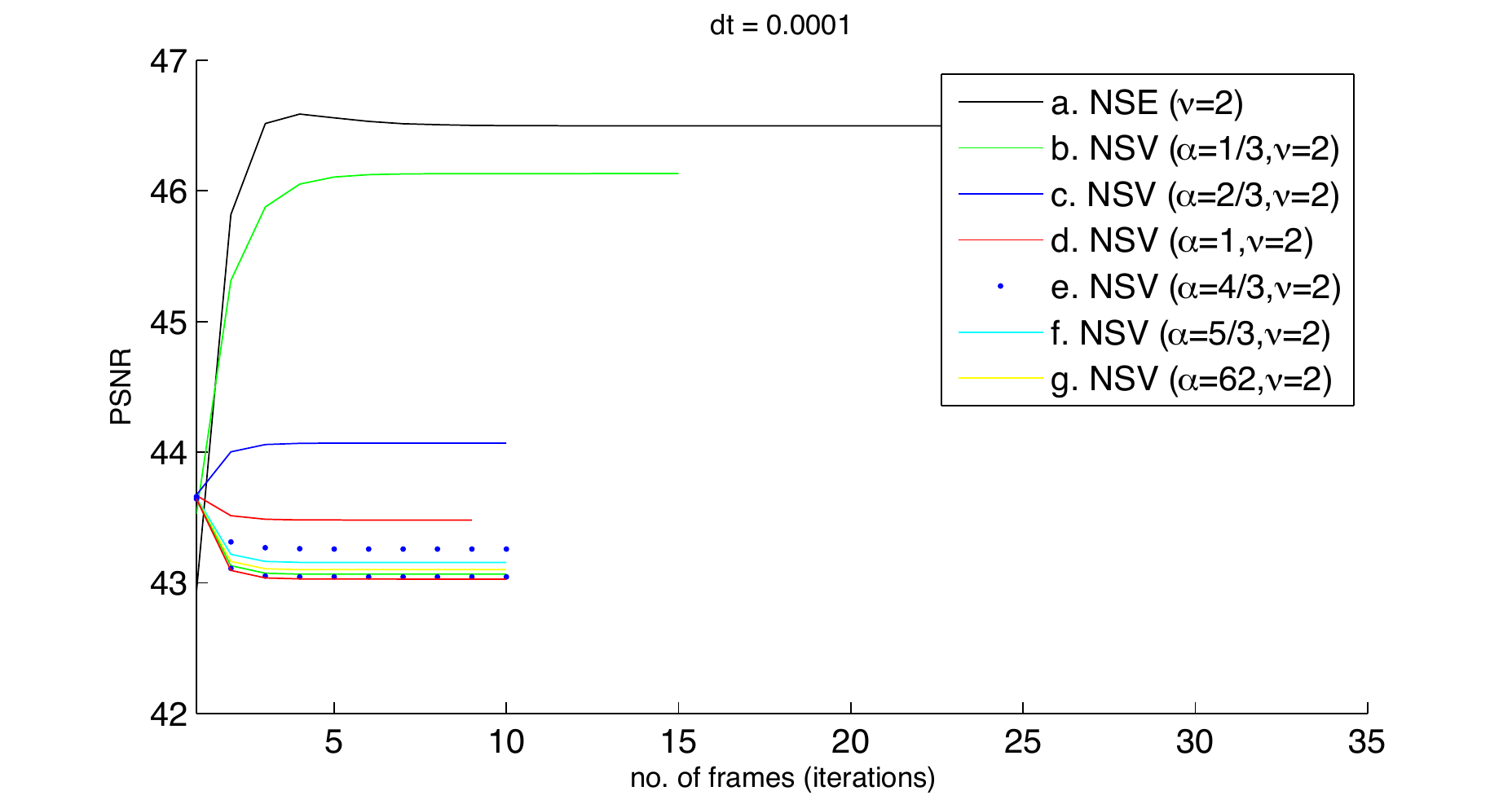}
}
\caption{\small\em PSNR of NSE vs NSV (for various values of $\alpha$).  When $dt=0.001$, the PSNR of NSV ($\alpha$=1/3) is higher than the NSE and requires only about half of the number of FLOPS needed for the NSE to converge to its steady state solution.}
\label{psnr}
\end{center}
\end{figure}
\begin{enumerate}
\item Comparing NSE $(\alpha=0)$ and NSV with $\alpha=1/3$ for both time steps $dt=0.001$ and $dt=0.0001$ in Table~\ref{psnr},  we notice the number of FLOPS needed by NSV to converge to steady state has been reduced to almost half of the FLOPS required by the NSE to converge to its steady state solution.
\item  In Figure~\ref{psnr}a, for smaller $dt =0.001$, the PSNR of NSV ($\alpha=1/2$) is 45 dB and is 1  point higher than that of NSE.  In the case when $dt=0.0001$, (Figure~\ref{psnr}b) the PSNR of NSE is 46.5 dB while the NSV for $\alpha=1/3$ has a PSNR value of 46.  
\item We also notice that increasing the time step from $dt=0.0001$ to $dt=0.001$ for the NSE changes the quality of the picture dramatically under subjective measure.  In Figure~\ref{psnr}, the PSNR for NSE is reduced from 46.5 dB to 44 dB when $dt$ is increased from $0.0001$ to $0.001$.  In the case of NSV for $\alpha=1/3$, the change in quality of the picture under subjective measure is unnoticeable.  In particular, the PSNR is changed from 46 dB to 45 dB as the $dt$ is increased. 
\end{enumerate}
In summary, one can show that for some value of $\alpha>0$, NSV gives an inpainting which is comparable to NSE but requiring less resources.
It is then natural to ask if one can determine the value of $\alpha$ {\em a priori} which will yield the smallest FLOP count with maximum PSNR.
Our preliminary studies suggest that choice of $\alpha$ may be image dependent.
For this reason, one may prefer the NSE to avoid finding the optimal $\alpha$.
This suggests the development of schemes to adaptively determine optimized $\alpha$.
% This will be a subject for future work.    

% \begin{comment}
% In Figures \ref{params1} and \ref{params2} we present some parameter study of two images with different noise which tells us that the choice of $\alpha$ may be image dependent and that the 2D NSE in most cases yields the maximum PSNR even if the resulting images are comparable under subjective measure.  For this reason, one may prefer the 2D NSE to avoid the difficulty in finding the optimal $\alpha$ for which the 2D NSV will yield the highest PSNR.

% We compute the PSNR and FLOP count for various $\alpha$ values and plot the results in Figure \ref{params1}.  In Figure \ref{params1}a, we show how the PSNR decreases as the $\alpha$ increases from 0 to 0.1.  The squares are for NSE ($\alpha=0$) with varying viscosity $.5 \leq\nu\leq 5$.  Figure \ref{params1}b shows the corresponding FLOP counts.  The FLOP count for the case when $\alpha=0.1$ is less than that for NSE and its PSNR differs only in less than 0.1 with that of NSE. In Figure \ref{params2}a, we see a similar trend.  The PSNR decreases as the $\alpha$ increases from 0 to 0.1.  The squares are for NSE ($\alpha=0$) with varying viscosity $.5 \leq\nu\leq 5$.  Figure \ref{params2}b shows the corresponding FLOP counts.  The FLOP count for the case when $\alpha=1/3$ is much higher than that of NSE and its PSNR is comparable to that of NSE.       
% \end{comment}

\begin{figure}[ht]
\begin{center}
\subfigure[]{
\includegraphics[scale=0.48]{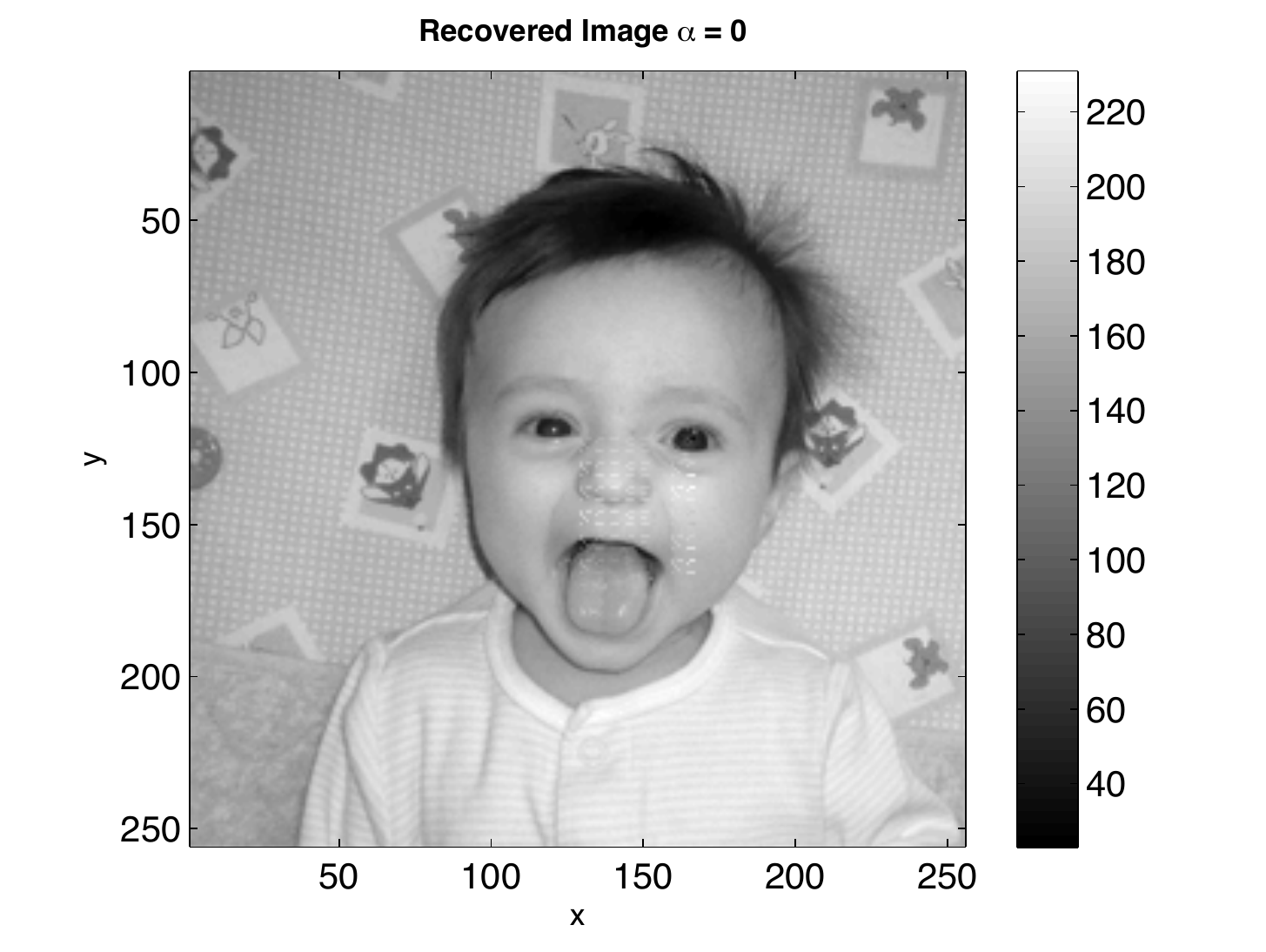}
}
\subfigure[]{
\includegraphics[scale=0.48]{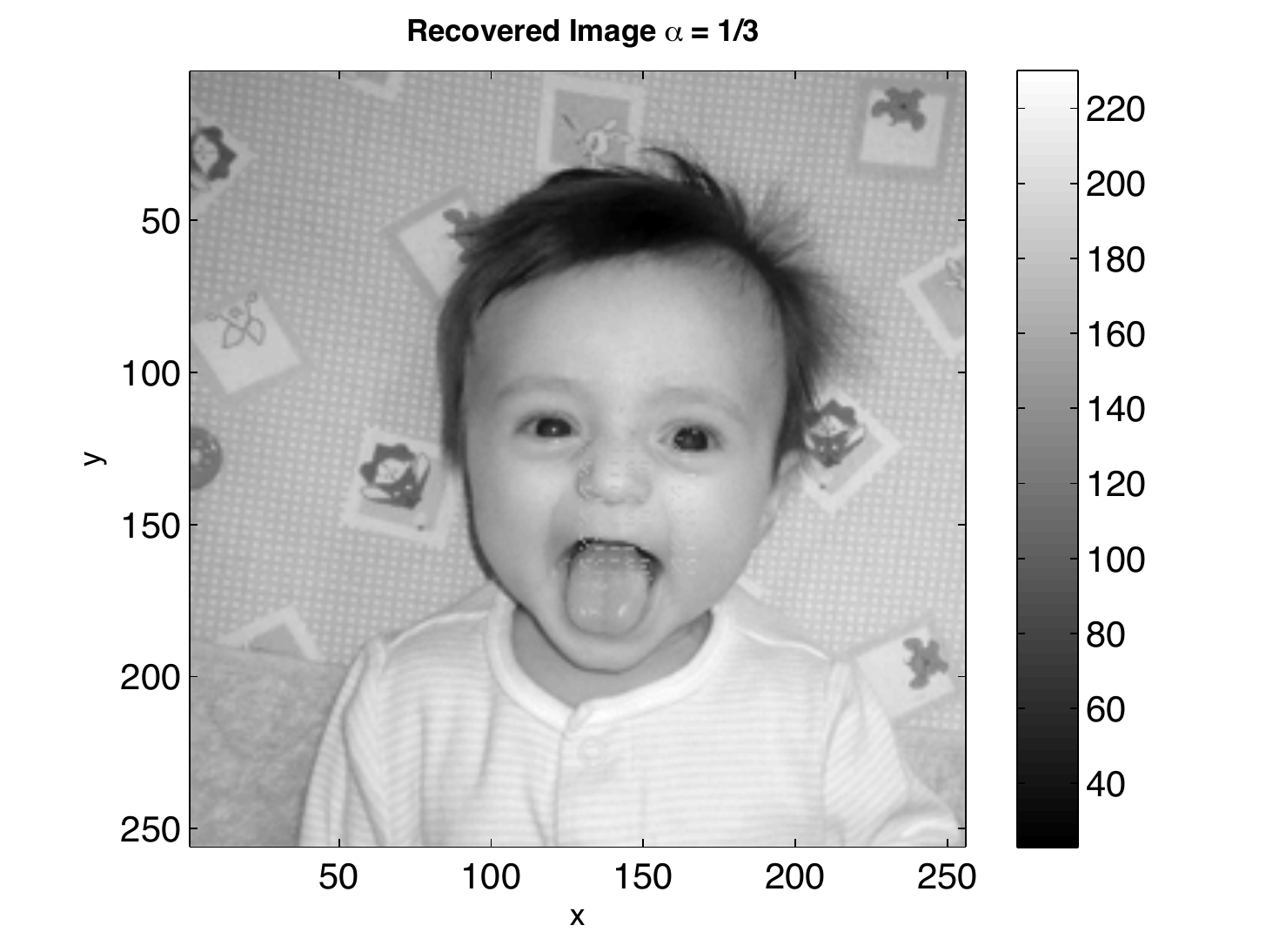}
}
%\subfigure[]{
%\includegraphics[scale=0.4]{103}
%}
%\subfigure[]{
%\includegraphics[scale=0.4]{104}
%}
%\subfigure[]{
%\includegraphics[scale=0.4]{105}
%}
\caption{\small\em An example of image recovered for (a) NSE  and (b) NSV ($\alpha=1/3$) with $dt=0.0001$.}
\label{101_105}
\end{center}
\end{figure}

%% file: theory.tex
\section{Some Supporting Theoretical Results}
\label{sec:theory}

In this section we summarize some existing results, and
then establish some new results that support the numerical results
in the previous section.
To this end, let $\Omega = [0,2\pi L]^2$.
The NSE of viscous incompressible flows, subject to periodic boundary condition on domain $\Omega$, is written in the form:
\begin{equation}
\aligned
\pp_t u - \nu\Dd u + (u\cdot\nabla)
 u &= -\nabla p + f,\\
 \nabla \cdot u &= 0,
\endaligned
\end{equation}
with initial condition $u(x,0) = u^{in}(x)$,
where $u$ represents the unknown fluid velocity,
$p$ is the unknown pressure scalar, $\nu >
0$ is the constant kinematic viscosity,
The given forcing function $f$ is assumed here to be
time independent and with mean zero:
$\int_\Omega f(x) dx = 0$.
The initial velocity $u^{in}$ is 
also assumed to have zero mean, hence also $u$.
Below we use will use standard notation in the context of the mathematical theory of Navier-Stokes equations (NSE) (see, e.g.,~\cite{CF88,Temam88,Temam01}).
In particular,
\begin{enumerate}
\item We denote by $L^p$ and $H^m$ the usual Lebesgue and Sobolev spaces,
respectively.
And we denote by $|\cdot|$ and $(\cdot,\cdot)$ the $L^2-$norm
and $L^2-$inner product, respectively.
\item Let $\mathcal{F}$ be the set of all vector trigonometric polynomials
with periodic domain $\Omega$.
We then set
$$
\mathcal{V}=\left\{\phi \in
\mathcal{F}:\nabla\cdot\phi = 0 \ \mbox{and}
\int_\Omega \phi(x)\ dx = 0\right\}.
$$
We set $H$ and $V$ to be the closures of $\mathcal{V}$ in $L^2$ and $H^1$,
respectively.
We note by Rellich lemma (see, e.g.,~\cite{AR75}) we have the $V$ is compactly embedded in $H$.
\item We denote by $P_\ss:L^2 \rightarrow H$ the Helmholtz-Leray orthogonal
projection operator, and by $A = -P_\ss\Dd$ the Stokes operator subject to
periodic boundary condition with domain $D(A) = (H^2(\Omega))^2\cap V$.
We note that in the space-periodic case,
\begin{equation*}
Au = -P_\ss\Dd u = -\Dd u, \hspace{.5cm} \mbox{for all }u \in D(A),
\end{equation*}
with $A^{-1}$ self-adjoint positive definite and compact from $H$ into $H$ (cf.~\cite{CF88,Temam01}).
Denote $0 < {L}^{-2}= \lambda_1 \leq \lambda_2 \leq \dots \dots$ the eigenvalues of $A$, repeated according to their multiplicities.  
\item We recall the following two-dimensional Ladyzhenskaya inequality:
\begin{equation}\label{Lady}
\aligned
&\|\phi\|_{L^4}\leq c\|\phi\|_{L^2}^{1/2}\|\phi\|_{H^1}^{1/2},
\hspace{.5cm}\mbox{for every }
 \phi \in H^1(\Omega).
\endaligned
\end{equation}
Hereafter $c$ will denote a generic dimensionless constant.

\item For $w_1, w_2 \in \mathcal{V}$, we define the bilinear form
\begin{equation}\label{B1}
B(w_1,w_2) = P_\ss((w_1\cdot\nabla)w_2).
\end{equation}
\item By the Sobolev inequalities and compactness theorems, in two dimensions (or any dimensions less than 4)  we can define on $V$ a trilinear continuous form $b$ by setting
\begin{equation}
b(u,v,w) = (B(u,v), w).  
\end{equation}
If $u, v \in V$ then
\begin{equation}\label{buuu}
b(u,v,v) =0. 
\end{equation}
A special case we will need to establish for stability of 2D NSV is:
\begin{equation}\label{buuAu}
b(u,u,Au) =0, \mbox{for all }u\in D(A).
\end{equation}

\end{enumerate}

We finish by mentioning some useful inequalities for semi-discretizations.
Let $n$ be the spatial dimension.
The following are easily established (cf.~\cite{Temam88,Temam01}):
\begin{eqnarray}\label{reverse_poincare_pf}
\aligned
\|u_h\|_h^2 &= \sum_{i,j=1}^n 
\left\{
\frac{1}{h_j^2}\int_\Omega \lvert u_{ih}\left(x+\frac{\vec{h}_j}{2}\right)-u_{ih}\left(x-\frac{\vec{h}_j}{2}\right)\lvert^2 \;dx
\right\}
\\
% &\leq 2 \sum_{i,j=1}^n 
% \left\{
% \frac{1}{h_j^2}\int_\Omega \lvert u_{ih}\left(x+\frac{\vec{h}_j}{2}\right)\lvert^2+\lvert u_{ih}\left(x-\frac{\vec{h}_j}{2}\right)\lvert^2 \;dx
% \right\}  \\
% &\leq 4  \sum_{i,j=1}^n 
% \left\{
% \frac{1}{h_j^2}\int_\Omega \lvert u_{ih}\left(x\right)\lvert^2 \;dx
% \right\}
&\leq 4 \left(\sum_{j=1}^n\frac{1}{h_j^2}\right)|u_h|^2 = S(h)^2|u_h|^2
\endaligned
\end{eqnarray}

\begin{eqnarray}\label{reverse_poincare_pf2}
\aligned
|\Delta_h u_h|^2 &= \dfrac{1}{4}\sum_{i,j=1}^n 
\left\{
\frac{1}{h_j^2}\int_\Omega \lvert u_{ih}\left(x+\vec{h}_j\right)+u_{ih}\left(x-\vec{h}_j\right) -2u_{ih}(x)\lvert^2\;dx
\right\} \\
% &\leq  \sum_{i,j=1}^n 
% \left\{
% \frac{1}{h_j^2}\int_\Omega \Big\lvert \dfrac{u_{ih}\left(x+\vec{h}_j\right)-u_{ih}(x)}{h_j}\Big\lvert^2
% +\Big\lvert \dfrac{u_{ih}\left(x-\vec{h}_j\right)-u_{ih}(x)}{h_j}\Big\lvert^2 \;dx \right\}  \\
% &\leq 2  \sum_{i,j=1}^n 
% \left\{
% \frac{1}{h_j^2}\int_\Omega \lvert \nabla_{jh}u_{ih}\lvert^2 \;dx
% \right\}
& \leq 4 \left(\sum_{j=1}^n\frac{1}{h_j^2}\right)\|u_h\|^2 = S(h)^2\|u_h\|^2
\endaligned
\end{eqnarray}

%%%%%%%%%%%%%%%%%%%%%%%%%%%%%%%%%%%%%%%%%%%%%%%%%%%%%%%%%%%%%%%%%%%%%%%%%%%%%%
%\section{Dependence of solution on boundary, size of domain, and viscosity.}

%%%%%%%%%%%%%%%%%%%%%%%%%%%%%%%%%%%%%%%
\subsection{Uniqueness of Steady State Solutions for Boundary Driven Flows}\label{subsec:uniqueness}

The results in~\cite{BBS01} is obtained by requiring the user to tune the parameters and related data in the inpainting region to suit the particular problem at hand.
In~\cite{ATXX} while inpainting  various sample images, with similar inpainting region, viscosity $\nu$ and $\delta t$, it was found that certain set of parameters produced stable solutions in some images and unstable solutions (gray-level blows up) in others.
They have noted that certain characteristics of $I$ near $\partial \Omega$ has some effect on the maximum allowable stable choice of $\delta t$.
In this section we will present analytical arguments on the dependence of the solution on the image at the boundary, size of the inpainting region and viscosity.
We give some hypothesis on how the related data affects the convergence of the numerical solution.
Note that the steady state solution for the NSV is exactly equal to the steady state solution of NSE.
The main goal for this study is to determine the relationship between the viscosity, the image at the boundary and the size of the inpainting region.
For example, we would like to determine for which size of the inpainting region and norm of $I$ on the boundary, do we get a unique steady state solution. 

In~\cite{BBS01}, for the Navier-Stokes based inpainting, a discussion on the uniqueness of steady state solution and its relevance to inpainting was presented.
It was expected that Navier-Stokes based inpainting may inherit some of the stability and uniqueness issues known for incompressible fluids, although the effect of anisotropic diffusion is still unclear.
The dependence of uniqueness in the viscosity of the fluid is discussed in~\cite{BBS01}.
In this section, we present some rigorous arguments following the work in~\cite{Temam88,Temam01}), modified slightly to interpret it in the context of image inpainting.
In this section we show the dependence of the uniqueness of steady solution on the viscosity, on the image at the boundary, and on the size of the inpainting region.         
We start here  by recalling some notation.
The notation used here are similar to those used in Section~\ref{NSV-section}.
Let us denote by $\Omega$ a bounded domain of $\mathbb{R}^2$ of class $\mathcal{C}^2$ which is filled with an incompressible viscous fluid.  

For the particular application of NSE and NSV in image inpainting, we supplement it with Dirichlet boundary condition $u=\phi$ on $\partial \Omega$, in which $\phi$ is independent of time.
We consider here the non-homogeneous steady state Navier-Stokes problem which coincides with the steady state Navier-Stokes-Voight problem.
Find $u$ and $p$ such that
\begin{eqnarray}\label{NSE-nonhomo}
-\nu\Delta u + (u\cdot\nabla) u + \nabla p =0, \quad \mbox{in }\Omega,\\
\mbox{div }u =0, \quad \mbox{in }\Omega,\\
u=\phi \quad \mbox{on } \partial\Omega.
\end{eqnarray}

We assume that $\phi$ is given as the trace on $\partial \Omega$ of a function $\Phi$,
\begin{equation}\label{bcond}
\Phi\in H^2(\Omega),\quad \mbox{div }\Phi =0, \quad \int_{\partial\Omega} \Phi\cdot \mbox{n}\; ds =0,
\end{equation}
where n is the unit outward normal in $\partial\Omega$.
The idea is the following:  Given the physical data $\phi$ defined on $\partial\Omega$, we find an extension $\Phi$ of $\phi$ inside $\Omega$ satisfying (\ref{bcond}).
Under the above hypothesis there exist at least one $u \in H^1$ and distribution $p$ on $\Omega$ satisfying (\ref{NSE-nonhomo}) (see~\cite{Temam88,Temam01}) provided we choose the extension $\Phi\in L^a(\Omega)$ of $\phi$ such that $\|\Phi\|_a$ is sufficiently small for $a>2$ so that 
\begin{equation}\label{smallness-cond}
|b(v,\Phi,v)|\leq \frac{\nu}{2}\|v\|^2,\quad \mbox{for all }   v\in V.
\end{equation}
  The construction of $\Phi$ satisfying (\ref{bcond}) and (\ref{smallness-cond}) is presented in~\cite{Temam01}.
Knowing $\Phi$ and knowing $u$ is a solution of (\ref{NSE-nonhomo}) then letting $\hat{u} = u-\Phi$, equation (\ref{NSE-nonhomo}) is equivalent to
\begin{eqnarray}\label{NSE-nonhomo2}
-\nu\Delta \hat{u} + (\hat{u}\cdot\nabla) \hat{u} +(\hat{u}\cdot\nabla)\Phi+ (\Phi\cdot\nabla)\hat{u}+ \nabla p =\nu\Delta\Phi - (\Phi\cdot\nabla) \Phi, \quad \mbox{in }\Omega,\\
\mbox{div }\hat{u} =0, \quad \mbox{in }\Omega,\\
\hat{u}=0 \quad \mbox{on } \partial\Omega.
\end{eqnarray}
We now state a uniqueness result.
\begin{theorem} Suppose that the norm of $\Phi$ in $L^a(\Omega)$ is sufficiently small so that
\begin{equation}
|b(v,\Phi,v)|\leq \frac{\nu}{2}\|v\|^2,\quad \mbox{for all }   v\in V,
\end{equation}
and $\nu$ is sufficiently large so that 
\begin{equation}
\nu^2 > 2C\lambda_1^{-1/2}\|f\|,
\end{equation}
where $f = \nu\Delta \Phi - B(\Phi,\Phi)$, $\lambda_1$ is the smallest eigenvalue of the Stokes operator, and $C$ is constant, then \pref{NSE-nonhomo} has a unique solution.
\end{theorem}
\begin{proof}
See~\cite{Temam88,Temam01}.
\end{proof}

We can recast the theorem above, in the context of image inpainting, as  the dependence of the uniqueness of steady solution on the viscosity $\nu$, on the image at the boundary, which is related to $\Phi$ (since the stream function is related to the velocity field), and on the size of the inpainting region which is related to $\lambda^{-1/2}$.  

%%%%%%%%%%%%%%%%%%%%%%%%%%%%%%%%%%%%%%%%%%%%%%%%%%%%%%%%%%%%%%%%%%%%%%%%%%%%%%
\subsection{Stability Analysis of the Scheme based on 2D NSE}\label{subsec:stab_anal_nse}

In this section we are concerned with a discussion on the discretization of the Navier Stokes equations in two dimensions, subject to periodic boundary conditions, with basic domain $\Omega = [0,2\pi L]^2$.
We study here a full discretization of the equations, both in space and time.
We start with a description of the approximating scheme and then proceed to the study of the stability of this scheme.
Our study here is based on the energy methods similar as in~\cite{Temam01} which leads to sufficient conditions for stability.

%%%%%%%%%%%%%%%%%%%%%%%%%%%%%%%%%%%%%%%
%\subsection{Description of the Approximation Scheme}

We now begin by defining a Galerkin approximation of the separable normed space $V$.
For reference we direct the reader to~\cite{Temam01}.
Let $V_h, \; h\in \N$, be an increasing sequence of finite-dimensional subspaces of $V$ whose union is dense in $V$.
For simplicity, assume that 
\begin{equation}
V_h\subset L^2(\Om), \quad\mbox{for all } h\in \N.
\end{equation}
The space $V_h$ is therefore equipped with two norms:  the norm $|\cdot|$ induced  by $L^2(\Om)$ and its own norm $\|\cdot\|_h$.
Since $V_h$ is finite dimensional then these two norms must be equivalent.
To be more precise we have with $d_0$ independent of $h$,
\begin{eqnarray}
|u_h| & \leq & d_0 \|u_h\|_h,\quad  \mbox{for all } u_h\in V_h,
\label{d0}
\\
\|u_h\|_h & \leq & S(h)|u_h|, \quad \mbox{for all } u_h\in V_h.
\label{reverse_poincare1}
\end{eqnarray}
Similarly, 
 let $D(A)_h, \; h\in \N$, be an increasing sequence of finite-dimensional subspaces of $D(A)$ whose union is dense in $D(A)$.
For simplicity, assume that 
\begin{equation}
D(A)_h\subset V_h\subset L^2(\Om), \quad\mbox{for all } h\in \N.
\end{equation}
The space $D(A)_h$ is therefore equipped with two norms:  the norm $\|\cdot\|_h$ induced  by $V_h$ and its own norm $|A_h\cdot|$.
Since $D(A)_h$ is finite dimensional normed space then these two norms must be equivalent.
To be more precise we have with $d_2$ independent of $h$,
\begin{eqnarray}
\|u_h\|_h & \leq & d_2|A_hu_h|,\quad  \mbox{for all } u_h\in D(A)_h,
\\
|A_hu_h| & \leq & c S(h)\|u_h\|_h, \quad \mbox{for all } u_h\in D(A)_h.
\label{reverse_poincare2}
\end{eqnarray}
The constant $S(h)$, which depends on $h$ is sometimes called the {\it stability constant} since it plays a major important role in obtaining necessary conditions on the stability of numerical approximations.
Usually $S(h)\rightarrow\infty$, as $h\rightarrow 0$.

Let there be given a trilinear continuous form on $V_h$, say $b_h(u_h,v_h,w_h)$ which satisfies the following:
\begin{enumerate}
\item For all  $u_h, v_h\in V_h,$ both of the following hold:
\begin{equation}
b_h(u_h,v_h,v_h) =0, 
\end{equation}
\begin{equation}\label{buuvh}
\aligned
|b_h(u_h,u_h,v_h)|
&
\leq d_1\|u_h\|_h^2\|v_h\|_h
%&
\leq d_1S^2(h)|u_h|\|u_h\|_h|v_h|
\\
&
\leq S_1(h)|u_h|\|u_h\|_h|v_h|,
\endaligned
\end{equation}
where at least, 
\begin{equation}
S_1(h)\leq d_1 S^2(h).
\end{equation}
\item For all  $u_h, v_h, w_h \in V_h,$
\begin{equation}
|b_h(u_h,v_h,w_h)|\leq d_1\|u_h\|_h\|v_h\|_h\|w_h\|_h. 
\end{equation}
\end{enumerate}
We divide the interval $[0,T]$ into $N$ intervals of equal length 
$k=T/N$.
We associate with $k$ and the function $f$, the elements $f_1, \dots, f^N$:
\begin{equation}
f^m = \frac{1}{k}\int^{mk}_{(m-1)k}f(t)\;dt, \quad m=1,\dots,N;\quad f^m\in L^2(\Om).
\end{equation}
We denote by $u_h^0$ the orthogonal projection of the initial condition $u_0$ onto $D(A)_h$ in $L^2(\Om)$.  
When $u_h^0,\dots,u_h^{m-1}$, are known, $u_h^m$ is the solution in $D(A)_h$ of
\begin{equation}\label{scheme}
 \frac{1}{k}(\uh{m}-\uh{m-1},v_h) + \nu((\uh{m-1},v_h))_h + b_h(\uh{m-1}, \uh{m-1},v_h) = (f_m,v_h)
\end{equation}
for all $v_h\in D(A)_h$.
  
%%%%%%%%%%%%%%%%%%%%%%%%%%%%%%%%%%%%%%%
%\subsection{Stability of the Scheme}
\begin{lemma}\label{lemma-nse}
We assume that $k$ and $h$ satisfy
\begin{enumerate}
\item $kS^2(h)\leq \dfrac{1-\delta}{4\nu}$ for some $\delta$, $0<\delta<1$ ,
\item $kS^2(h)\leq 1$,
\item $kS_1^2(h)S^2(h)\leq \dfrac{\nu\delta}{8d_0^2d_5}$,
\end{enumerate}
where $d_0$ is as in (\ref{d0}) and 
\begin{equation}
d_5=\|u_0\|^2 + d_0^2\left(\dfrac{d_0^2+1-\delta}{\nu d_0^2}\right)\int_0^T|f(s)|^2 \;ds,
\end{equation}
then, the $\uh{m}$ given by \pref{scheme} remain bounded in the following sense
\begin{equation}\label{bd1-nse}
\|\uh{m}\|_h^2\leq d_5 \quad m=1,\dots,N,
\end{equation}
\begin{equation}\label{bd2-nse}
k\sum_{m=1}^r|A_h\uh{m-1}|^2\leq\dfrac{2d_5}{\nu\delta},
\end{equation}
\begin{equation}\label{bd3-nse}
\sum_{m=1}^N\|\uh{m}-\uh{m-1}\|_h^2\leq 2\left(\dfrac{2-\delta}{\delta}\right)d_5 + 4 \int_0^T|f(s)|^2\;ds.
\end{equation}
\end{lemma}
\begin{proof}
We replace by $v_h$ by $A_h\uh{m-1}$ in \pref{scheme}; due to the identity
\begin{equation}\label{L2-identity}
2(a-b,b)=|a|^2-|b|^2-|a-b|^2,
\end{equation} 
we find
\begin{equation}\label{eq1pt9}
\aligned
\Vnorm{\uh{m}}^2-\Vnorm{\uh{m-1}}^2 -\Vnorm{\uh{m}-\uh{m-1}}^2 
&+ 2k\nu|A_h\uh{m-1}|^2 
= 2k(f^m, A_h\uh{m-1})
\\
& \leq 2kd_0|f^m||A_h\uh{m-1}| \\
& \leq k\nu|A_h\uh{m-1}|^2 + k\dfrac{d_0^2}{\nu}|f^m|^2.
\endaligned
\end{equation}
% that is,
% \begin{equation}\label{eq1pt9}
% \aligned
% \Vnorm{\uh{m}}^2-\Vnorm{\uh{m-1}}^2 -\Vnorm{\uh{m}-\uh{m-1}}^2 + 2k\nu|A_h\uh{m-1}|^2 
% &= 2k(f^m, A_h\uh{m-1})\\
% &\leq kd_0|f^m||A_h\uh{m-1}|\\
% &\leq k\dfrac{d_0^2}{\nu}|f^m|^2
% \endaligned
% \end{equation}
We would like to majorize $\Vnorm{\uh{m}-\uh{m-1}}^2$ in \pref{eq1pt9}.
Let $v_h = A_h\uh{m}-A_h\uh{m-1}$ in \pref{scheme}.
This gives
\begin{equation}\label{I123}
\aligned
2\Vnorm{\uh{m}-\uh{m-1}}^2 &= -2k\nu\paren{A_h\uh{m-1},A(\uh{m}-\uh{m-1})} 
\\
& \qquad - 2kb_h(\uh{m-1},\uh{m-1}, A_h{\uh{m}-\uh{m-1}}) 
\\
& \qquad + 2k(f^m, A_h(\uh{m}-\uh{m-1}))
=:I_1 + I_2 + I_3.
\endaligned
\end{equation}
We successively majorize $I_1$, $I_2$ and $I_3$ using \pref{reverse_poincare1}, \pref{reverse_poincare2}, \pref{buuvh}, and Cauchy-Schwarz inequality
\begin{eqnarray}
|I_1| &\leq & 2k\nu|A_h\uh{m-1}| |A_h(\uh{m}-\uh{m-1})|
\\
&\leq& 2k\nu S(h)|A_h\uh{m-1}|\Vnorm{\uh{m}-\uh{m-1}}
\\
&\leq& \dfrac{1}{4}\Vnorm{\uh{m}-\uh{m-1}}^2 + 4k^2\nu^2S^2(h)|A_h\uh{m-1}|^2,
\label{I1}
\end{eqnarray}
\begin{eqnarray}
|I_2| &\leq& 2kS_1(h)|\uh{m-1}|\Vnorm{\uh{m-1}} |A_h(\uh{m}-\uh{m-1})|
\\
&\leq& 2kS(h)S_1(h)|\uh{m-1}|\Vnorm{\uh{m-1}} \Vnorm{\uh{m}\uh{m-1}}
\\
&\leq&  \dfrac{1}{4}\Vnorm{\uh{m}-\uh{m-1}}^2 + 4k^2d_0S^2(h)S_1^2(h)|\uh{m-1}|^2\Vnorm{\uh{m-1}}^2
\\
&\leq& \dfrac{1}{4}\Vnorm{\uh{m}-\uh{m-1}}^2 + 4k^2d_0^2S^2(h)S_1^2(h)|\uh{m-1}|^2|A_h\uh{m-1}|^2,
\label{I2}
\end{eqnarray}
\begin{eqnarray}
|I_3|&\leq& 2k|f^m||A_h(\uh{m} - \uh{m-1})|
\leq 2k S(h) |f^m| \Vnorm{\uh{m} - \uh{m-1}}
\\
&\leq& \dfrac{1}{4}\Vnorm{\uh{m}-\uh{m-1}}^2 + 4k^2S^2(h)|f^m|.
\label{I3}
\end{eqnarray}
Denoting as $\Theta = 4k^2d_0^2S_1^2(h)S^2(h)|\uh{m-1}|^2|A_h\uh{m-1}|^2$,
we have that \pref{I123} becomes
\begin{equation}\label{eq1pt11}
\aligned
\Vnorm{\uh{m}-\uh{m-1}}^2&\leq 4k^2\nu^2S^(h) |A_h\uh{m-1}|^2
+ \Theta
% + 4k^2d_0^2S_1^2(h)S^2(h)|\uh{m-1}|^2|A_h\uh{m-1}|^2 + 4k^2S^2(h)|f^m|^2\\
\leq k\nu (1-\delta)|A_h\uh{m-1}|^2 
+ \Theta.
% + 4k^2d_0^2S_1^2(h)S^2(h)|\uh{m-1}|^2|A_h\uh{m-1}|^2 + 4k^2S^2(h)|f^m|^2
\endaligned
\end{equation}
Going back to \pref{eq1pt9} we have
\begin{equation}\label{eq1pt12}
\aligned
\Vnorm{\uh{m}}^2-\Vnorm{\uh{m-1}}^2 &+ k(\nu\delta-4kd_0^2S_1^2(h)S^2(h)|\uh{m-1}|^2) |A_h\uh{m-1}|^2
\\
& \leq k\left(\dfrac{d_0^2}{\nu}+4kS^2(h)\right)|f^m|^2
\\
&\leq kd_0^2\left(\dfrac{1}{\nu}+\dfrac{1-\delta}{\nu d_0^2}\right)|f^m|^2
\\
&\leq kd_0^2\left(\dfrac{d_0^2+1-\delta}{\nu d_0^2}\right)|f^m|^2.
\endaligned
\end{equation}
Summing up \pref{eq1pt12} for $m=1$ to $r$ we get
\begin{equation}\label{eq1pt13}
\Vnorm{\uh{r}}^2 + k\sum_{m=1}^r\left(\nu\delta-4kd_0^2S_1^2(h)S^2(h)|\uh{m-1}|^2\right)|A_h\uh{m-1}|^2\leq \mu_r,
\end{equation}
where
$$\mu_r = \Vnorm{\uh{0}}^2 + kd_0^2\left(\dfrac{d_0^2+1-\delta}{\nu d_0^2}\right)\sum_{m=1}^r|f^m|^2.$$
Now using condition $(iii)$ in Lemma \ref{lemma-nse}, we will prove by the method of induction that
\begin{equation}\label{eq1pt14}
\Vnorm{\uh{r}}^2 + \dfrac{k\nu\delta}{2}\sum_{m=1}^r|A_h\uh{m-1}|^2 \leq \mu_r,\quad r=1 \mbox{ to } N.
\end{equation}
First observe that
\begin{equation}
\aligned
\mu_r\leq\mu_N &= \|\uh{0}\|^2 + kd_0^2\left(\dfrac{d_0^2+1-\delta}{\nu d_0^2}\right)\sum_{m-1}^N|f^m|^2\\
&\leq \|u_0\|^2 + d_0^2\left(\dfrac{d_0^2+1-\delta}{\nu d_0^2}\right)\int_0^T|f(s)|^2\;ds =: d_5 
\endaligned
\end{equation}
To establish the basis for induction $(r=1)$ we write \pref{eq1pt12} for $m=1$ and use $(iii)$ in Lemma \ref{lemma-nse} to get
\begin{equation}
\aligned
\Vnorm{\uh{1}}^2 + k\nu\delta |A_h\uh{0}|^2 &\leq \Vnorm{\uh{0}}^2 + 4k^2d_0^2S_1^2(h)S^2(h)|\uh{0}|^2|A_h\uh{0}|^2
\\
& \qquad + kd_0^2\left(\dfrac{d_0^2+1-\delta}{\nu d_0^2}\right)|f^1|^2\\
&\leq \mu_1 + \dfrac{k\nu\delta}{2}|A_h\uh{0}|^2.
\endaligned
\end{equation}
Thus equation \pref{eq1pt14} for $r=1$ is satisfied.
By induction on $r$, assume that \pref{eq1pt14} holds up to the order $r-1$.
Note that by the recurrence hypothesis
\begin{equation}\label{eq1pt17}
\Vnorm{\uh{r-1}}^2\leq\mu_{r-1}\leq \mu_N \leq d_5.
\end{equation}
Thus by \pref{eq1pt13}, we have
\begin{equation}
\aligned
\Vnorm{\uh{r}}^2 + k\nu\delta\sum_{m=1}^r|A_h\uh{m-1}|^2 &\leq \mu_r+ 4k^2d_0^2S_1^2(h) S^2(h)|\uh{m-1}|^2\\
&\leq \mu_r+ 4k^2d_0^2S_1^2(h) S^2(h)d_5 \sum_{m=1}^r |A_h\uh{m-1}|^2\\
&\leq \mu_2 + \dfrac{k\nu\delta}{2} \sum_{m=1}^r |A_h\uh{m-1}|^2.
\endaligned
\end{equation}
Hence,
\begin{equation}
\Vnorm{\uh{r}}^2 + \dfrac{k\nu\delta}{2} \sum_{m=1}^r |A_h\uh{m-1}|^2\leq \mu_r.
\end{equation}
This gives \pref{bd2-nse}.
It remains to prove \pref{bd3-nse}.
From \pref{eq1pt11}, applying $(ii)$ and $(iii)$ from Lemma \ref{lemma-nse}, we get
\begin{equation}
\aligned
\Vnorm{\uh{m}-\uh{m-1}}^2&\leq k^2\nu^2S^(h) |A_h\uh{m-1}|^2 + 4k^2d_0^2S_1^2(h)S^2(h)|\uh{m-1}|^2|A_h\uh\
{m-1}|^2 \\
& \qquad + 4k^2S^2(h)|f^m|^2\\
&\leq k\nu (1-\delta)|A_h\uh{m-1}|^2 + 4k \dfrac{\nu\delta}{8d_5}|\uh{m-1}|^2|A_h\uh{m-1}|^2 + 4k|f^m|^2\\
&\leq k\nu (1-\delta)|A_h\uh{m-1}|^2 + k \nu\delta|A_h\uh{m-1}|^2 + 4k|f^m|^2\\
&\leq k\nu (2-\delta)|A_h\uh{m-1}|^2 + 4k|f^m|^2.
\
\endaligned
\end{equation}
Summing it up over all $m=1$ to $N$ and using \pref{bd2-nse}
we find \pref{bd3-nse}.
\end{proof}

%%%%%%%%%%%%%%%%%%%%%%%%%%%%%%%%%%%%%%%%%%%%%%%%%%%%%%%%%%%%%%%%%%%%%%%%%%%%%%
\subsection{Stability Analysis of the Scheme based on 2D NSV}\label{subsec:stab_anal_nsv}
The preliminary setup is the same as those of 2D NSE in the previous section.
When $u_h^0,\dots,u_h^{m-1}$, are known, $u_h^m$ is the solution in $D(A)_h$ of
\begin{equation}\label{scheme-nsv}
\begin{split}
 \frac{1}{k}(\uh{m}-\uh{m-1},v_h) + \dfrac{\aa^2}{k}\paren{A_h\uh{m}-A_h\uh{m-1},v_h}
+ \nu((\uh{m-1},v_h))_h \\
+ b_h(\uh{m-1}, \uh{m-1},v_h) = (f_m,v_h),
\end{split}
\end{equation}
for all $v_h\in V_h$.

%%%%%%%%%%%%%%%%%%%%%%%%%%%%%%%%%%%%%%%
%\subsection{Stability of the Scheme}
\begin{lemma}\label{lemma-nsv}
We assume that $k$ and $h$ satisfy
\begin{enumerate}
\item $kS^2(h)\leq \dfrac{1-\delta}{4\nu}$ for some $\delta$, $0<\delta<1$,
\item $kS_1^2(h)\leq \dfrac{\nu\delta}{8d_6}$,
\end{enumerate}
\begin{equation}
d_6=|u_0|^2+\aa^2\|u_0\|^2 + \paren{\dfrac{d_0^2}{\nu}+4T}\int_0^T|f(s)|^2 \;ds,
\end{equation}
and $d_0$ is as in (\ref{d0}),
then, the $\uh{m}$ given by \pref{scheme-nsv} remain bounded in the following sense
\begin{equation}\label{bd1-nsv}
|\uh{m}|^2+\aa^2\|\uh{m}\|_h^2\leq d_6 \quad m=1,\dots,N,
\end{equation}
\begin{equation}\label{bd2-nsv}
k\sum_{m=1}^r\Vnorm{\uh{m-1}}^2\leq\dfrac{2d_6}{\nu\delta},
\end{equation}
\begin{equation}\label{bd3-nsv}
\sum_{m=1}^N\paren{|\uh{m}-\uh{m-1}|^2+\aa^2\|\uh{m}-\uh{m-1}\|_h^2}\leq \left(\dfrac{2-\delta}{\delta}\right)d_6 + 4T \int_0^T|f(s)|^2\
\;ds.
\end{equation}
\end{lemma}
\begin{proof}
We replace by $v_h$ by $\uh{m-1}$ in \pref{scheme-nsv}; due to \pref{d0}, \pref{reverse_poincare1},\pref{L2-identity}  and Cauchy-Schwarz, we get
\begin{equation}
\begin{split}
|\uh{m}|^2-|\uh{m-1}|^2 & -|\uh{m}-\uh{m-1}|^2 + 2k\nu\Vnorm{\uh{m-1}}^2
\\
& \qquad +\aa^2\paren{\Vnorm{\uh{m}}^2-\Vnorm{\uh{m-1}}^2 -\Vnorm{\uh{m}-\uh{m-1}}^2 }\\
&= 2k(f^m, A_h\uh{m-1})
\\
& \leq 2kd_0|f^m|\Vnorm{\uh{m-1}}
\leq k\nu\Vnorm{\uh{m-1}}^2 + k\dfrac{d_0^2}{\nu}|f^m|^2,
\end{split}
\end{equation}
that is,
\begin{equation}\label{eq2pt8}
\begin{split}
\paren{|\uh{m}|^2-|\uh{m-1}|^2} &+\aa^2\paren{\Vnorm{\uh{m}}^2-\Vnorm{\uh{m-1}}^2} \\
& \qquad -\paren{|\uh{m}-\uh{m-1}|^2+\aa^2\Vnorm{\uh{m}-\uh{m-1}}^2} 
\\
& \leq k\dfrac{d_0^2}{\nu}|f^m|^2.
\end{split}
\end{equation}
We would like to majorize the term $\paren{|\uh{m}-\uh{m-1}|^2+\aa^2\Vnorm{\uh{m}-\uh{m-1}}^2}$ in \pref{eq2pt8}.
Let $v_h = \uh{m}-\uh{m-1}$ in \pref{scheme-nsv}.
This gives
\begin{equation}\label{I123-nsv}
\aligned
& 2|\uh{m}-\uh{m-1}|^2+2\aa^2\Vnorm{\uh{m}-\uh{m-1}}^2 \\ 
& = -2k\nu\left(\paren{\uh{m-1},\uh{m}-\uh{m-1}}\right)_h \\
& \qquad - 2kb_h(\uh{m-1},\uh{m-1},{\uh{m}-\uh{m-1}})
+ 2k(f^m, \uh{m}-\uh{m-1})
=:I_1 + I_2 + I_3.
\endaligned
\end{equation}
We successively majorize $I_1$, $I_2$ and $I_3$ using repeatedly \pref{reverse_poincare1}, \pref{buuvh}, and Cauchy-Schwarz inequality and get exactly the estimates as in page 234 of \cite{Temam01}.  Therefore
%\begin{eqnarray}
%|I_1| &\leq& 2k\nu\Vnorm{\uh{m-1}} \Vnorm{\uh{m}-\uh{m-1}}
%\leq 2k\nu S(h)\Vnorm{\uh{m-1}}|\uh{m}-\uh{m-1}|
%\\
%&\leq& \dfrac{1}{4}|\uh{m}-\uh{m-1}|^2 + 4k^2\nu^2S^2(h)\Vnorm{\uh{m-1}}^2,
%\label{I1-nsv}
%\end{eqnarray}
%\begin{eqnarray}
%|I_2| &\leq& 2kS_1(h)|\uh{m-1}|\Vnorm{\uh{m-1}} |\uh{m}-\uh{m-1}|
%\\
%& \leq & \dfrac{1}{4}|\uh{m}-\uh{m-1}|^2 + 4k^2S_1^2(h)|\uh{m-1}|^2\Vnorm{\uh{m-1}}^2,
%\label{I2-nsv}
%\end{eqnarray}
%\begin{eqnarray}
%|I_3| &\leq & 2k|f^m||\uh{m} - \uh{m-1}|
%\leq \dfrac{1}{4}|\uh{m}-\uh{m-1}|^2 + 4k^2|f^m|.
%\label{I3-nsv}
%\end{eqnarray}
\pref{I123-nsv} becomes
\begin{equation}\label{eq2pt10}
\aligned
|\uh{m}-\uh{m-1}|^2+\aa^2\Vnorm{\uh{m}-\uh{m-1}}^2
&\leq 4k^2\nu^2S^(h) \Vnorm{\uh{m-1}}^2 \\
& + 4k^2S_1^2(h)|\uh{m-1}|^2\Vnorm{\uh{m-1}}^2 + 4k^2|f^m|^2\\
&\leq k\nu (1-\delta)\Vnorm{\uh{m-1}}^2 \\
& + 4k^2S_1^2(h)|\uh{m-1}|^2\Vnorm{\uh{m-1}}^2 + 4k^2|f^m|^2.
\endaligned
\end{equation}
Going back to \pref{eq2pt8} we have
\begin{equation}\label{eq2pt11}
\aligned
|\uh{m}|^2-|\uh{m-1}|^2 &+\aa^2\paren{\Vnorm{\uh{m}}^2-\Vnorm{\uh{m-1}}^2} 
\\
&\qquad + k(\nu\delta-4kS_1^2(h)|\uh{m-1}|^2) \Vnorm{\uh{m-1}}^2\\
&\leq k\left(\dfrac{d_0^2}{\nu}+4k\right)|f^m|^2\\
& \leq k\left(\dfrac{d_0^2}{\nu}+4T\right)|f^m|^2 \quad  \mbox{(since $k\leq T$).}
\endaligned
\end{equation}
Summing up \pref{eq2pt10} for $m=1$ to $r$ we get
\begin{equation}\label{eq2pt12}
|\uh{r}|^2+\aa^2\Vnorm{\uh{r}}^2 + k\sum_{m=1}^r\left(\nu\delta-4kS_1^2(h)|\uh{m-1}|^2\right)\Vnorm{\uh{m-1}}^2\leq \mu_r,
\end{equation}
where 
$$\mu_r = |\uh{0}|^2+\aa^2\Vnorm{\uh{0}}^2 + k\left(\dfrac{d_0^2}{\nu}+4T\right)\sum_{m=1}^r|f^m|^2.$$

Now using condition $(ii)$ in Lemma \ref{lemma-nsv}, we will prove by the method of induction that
\begin{equation}\label{eq2pt13}
|\uh{r}|^2+\aa^2\Vnorm{\uh{r}}^2 + \dfrac{k\nu\delta}{2}\sum_{m=1}^r\Vnorm{\uh{m-1}}^2 \leq \mu_r,\quad r=1 \mbox{ to } N.
\end{equation}

First observe that
\begin{equation}
\aligned
\mu_r\leq\mu_N &= |\uh{0}|^2+\aa^2\|\uh{0}\|^2 + k\left(\dfrac{d_0^2}{\nu } + 4T\right)\sum_{m-1}^N|f^m|^2 =: d_6.
\endaligned
\end{equation}
To establish the basis for induction $(r=1)$ we write \pref{eq2pt11} for $m=1$ and use $(ii)$ in Lemma \ref{lemma-nsv} to get
\begin{equation}
\aligned
& |\uh{1}|^2+\aa^2\Vnorm{\uh{1}}^2 + k\nu\delta \Vnorm{\uh{0}}^2
\leq |\uh{0}|^2+\aa^2\Vnorm{\uh{0}}^2, \\
& \qquad + 4k^2S_1^2(h)|\uh{0}|^2\Vnorm{\uh{0}}^2 + k\left(\dfrac{d_0^2}{\nu}+4T\right)|f^1|^2
\leq \mu_1 + \dfrac{k\nu\delta}{2}\Vnorm{\uh{0}}^2.
\endaligned
\end{equation}
Thus equation \pref{eq2pt13} for $r=1$ is satisfied.
By induction on $r$, assume that \pref{eq2pt13} holds up to the order $r-1$.
Note that by the recurrence hypothesis
\begin{equation}\label{eq2pt16}
|\uh{r-1}|^2+\aa^2\Vnorm{\uh{r-1}}^2\leq\mu_{r-1}\leq \mu_N \leq d_6.
\end{equation}
Thus by \pref{eq2pt12}, we have
\begin{equation}
\aligned
|\uh{r-1}|^2+\aa^2\Vnorm{\uh{r}}^2 + k\nu\delta\sum_{m=1}^r\Vnorm{\uh{m-1}}^2 &\leq \mu_r+ 4k^2S_1^2(h) |\uh{m-1}|^2|\uh{m-1}|^2\\
&\leq \mu_r+ 4k^2S_1^2(h) d_6 \sum_{m=1}^r \Vnorm{\uh{m-1}}^2\\
&\leq \mu_2 + \dfrac{k\nu\delta}{2} \sum_{m=1}^r \Vnorm{\uh{m-1}}^2.
\endaligned
\end{equation}
Hence,
\begin{equation}
|\uh{r}|^2+\aa^2\Vnorm{\uh{r}}^2 + \dfrac{k\nu\delta}{2} \sum_{m=1}^r \Vnorm{\uh{m-1}}^2\leq \mu_r.
\end{equation}
This gives us \pref{bd2-nsv}.
It remains to prove \pref{bd3-nsv}.
From \pref{eq2pt10} and applying condition $(i)$ and $(ii)$ from Lemma~\ref{lemma-nsv}, we get
\begin{equation}
\aligned
|\uh{m}-\uh{m-1}|^2+\aa^2\Vnorm{\uh{m}-\uh{m-1}}^2&\leq k\nu (1-\delta)\Vnorm{\uh{m-1}}^2 \\
& + 4k^2S_1^2(h)|\uh{m-1}|^2\Vnorm{\uh{m-1}}^2+4k^2 |f^m|^2\\
&\leq k\nu \left(1-\dfrac{\delta}{2}\right)\Vnorm{\uh{m-1}}^2 + 4kT|f^m|^2.
\endaligned
\end{equation}
Summing it up over all $m=1$ to $N$ and using \pref{bd2-nsv}
we find \pref{bd3-nsv}.
\end{proof}

%%%%%%%%%%%%%%%%%%%%%%%%%%%%%%%%%%%%%%%%%%%%%%%%%%%%%%%%%%%%%%%%%%%%%%%%%%%%%%
\subsection{Stability Analysis of a Semi-Implicit Scheme based on 2D NSE}\label{subsec:stab_anal_nse_semi}

We present here general results for stability analysis for the full NSE equations vorticity formulation  (implicit only on the linear part of the operator).
The setup is as follows:
When $u_h^0,\dots,u_h^{m-1}$, are known, $u_h^m$ is the solution in $D(A)_h$ of
\begin{equation}\label{scheme-nsesi}
 \frac{1}{k}(\uh{m}-\uh{m-1},v_h) + \nu((\uh{m},v_h))_h + b_h(\uh{m-1}, \uh{m-1},v_h) = (f_m,v_h),
\end{equation}
for all $v_h\in V_h$.

%%%%%%%%%%%%%%%%%%%%%%%%%%%%%%%%%%%%%%%
%\subsection{Stability of the Scheme}
\begin{lemma}\label{lemma-nsesi}
We assume that $k$ and $h$ satisfy
\begin{enumerate}
\item $kS^4(h)\leq d'$,
\item $kS_1^2(h)S^2(h)\leq d''$,
\end{enumerate}
where $d'=$ and $d''=$, then, the $\uh{m}$ given by \pref{scheme-nsesi} remain bounded in the following sense
\begin{equation}
\aligned
\Vnorm{\uh{m}}\leq d_7, \quad m &= 0,\dots, N,
\\
\sum_{m=1}^{N} \Vnorm{\uh{m}-\uh{m-1}}^2 \leq d_7,
&
\quad
\quad
k\sum_{m=1}^{N} |A_h\uh{m}|^2\leq d_7,
\endaligned
\end{equation}
where $d_7$ is some constant depending only on the data $d'$ and $d''$.
\end{lemma}
\begin{proof}
We let $v_h = A_h\uh{m}$ in \pref{scheme-nsesi}.
Using the identity \pref{L2-identity}, we get
\begin{equation}
\aligned
\Vnorm{\uh{m}}^2-\Vnorm{\uh{m-1}} &+ \Vnorm{\uh{m}-\uh{m-1}}^2 + 2k\nu|A_h\uh{m}|^2 \\
&= -2kb_h\left(\uh{m-1},\uh{m-1},A_h\uh{m}\right)
+ 2k(f^m, A_h\uh{m})\\
&:= I_1 + I_2.
\endaligned
\end{equation}
We would like bound the terms $I_1$ and $I_2$.
The identity $b_h\left(\uh{m-1},\uh{m-1},A\uh{m-1}\right)=0$ 
together with \pref{d0} and \pref{buuvh} gives
\begin{equation}
\aligned
|I_1| &\leq 2kS_1(h)|\uh{m-1}|\Vnorm{\uh{m-1}}|A(\uh{m}-\uh{m-1})|\\
&\leq 2k S_1(h)S(h)|\uh{m-1}|\Vnorm{\uh{m-1}}\Vnorm{\uh{m}-\uh{m-1}}\\
&\leq 2k^2S_1^2(h) S^2(h) |\uh{m-1}|^2\Vnorm{\uh{m-1}}^2 + \dfrac{1}{2}\Vnorm{\uh{m}-\uh{m-1}},
\endaligned
\end{equation}
and
\begin{equation}
|I_2| \leq 2kd_0|f^m||A_h\uh{m}|
\leq k \left(\dfrac{d_0^2}{\nu}|f_m|^2+\nu|A\uh{m}|^2\right).
\end{equation}
Hence,
\begin{equation}\label{5pt48}
\begin{split}
\Vnorm{\uh{m}}^2-\Vnorm{\uh{m-1}}^2 & + \frac{1}{2}\Vnorm{\uh{m}-\uh{m-1}}^2 + k\nu|A_h\uh{m}|^2 \\
& -2k^2S_1^2(h)S^2(h)|\uh{m-1}|^2\Vnorm{\uh{m-1}}^2\leq \dfrac{kd_0^2}{\nu}|f^m|^2.
\end{split}
\end{equation}
We sum up these inequalities for $m=1$ to $r$, to get
\begin{equation}\label{5pt49}
\begin{split}
\Vnorm{\uh{r}}^2 & + \frac{1}{2}\sum_{m=1}^r\Vnorm{\uh{m}-\uh{m-1}}^2 + k\nu\sum_{m=1}^r |A_h\uh{m}|^2 \\
& \qquad -2k^2S^2_1(h) S^2(h)\sum_{m=2}^r|\uh{m-1}|^2\Vnorm{\uh{m-1}}^2\leq \lambda_r,
\end{split}
\end{equation}
where,
\begin{equation}
\lambda_r = \Vnorm{\uh{0}} = \dfrac{kd_0^2}{\nu}\sum_{m=1}^r |f^m|^2 + 2k^2S^2_1(h) S^2(h)|\uh{0}|^2\Vnorm{\uh{0}}^2.
\end{equation}
We assume that 
\begin{equation}\label{5pt51}
2kd_0d_2S_1^2(h)S^2(h)\lambda_N \leq \nu-\delta,
\end{equation}
for some fixed $\delta$, $0\leq\delta\leq\nu$.
If this holds then one can show recursively that
\begin{equation}\label{5pt52}
\Vnorm{\uh{r}}^2 + \frac{1}{2}\sum_{m-1}^r\Vnorm{\uh{m}-\uh{m-1}}^2+k\delta\sum_{m=1}^r|A_h\uh{m}|^2\leq \lambda_r, \quad r=1,\dots, N.
\end{equation}
Clearly, letting $m=1$ in \pref{5pt48} shows that \pref{5pt52} is  true for $r=1$.
Let us assume that \pref{5pt52} is valid up to the order $r-1$, we want to show that \pref{5pt52} is valid for $r$.
Observe that by the inductive hypothesis $\Vnorm{\uh{m}}^2\leq\lambda_m\leq\lambda_N$.
Hence, by the condition \pref{5pt51}
\begin{equation}
\aligned
2k^2S_1^2(h)S^2(h)\sum_{m=2}^r|\uh{m-1}|\Vnorm{\uh{m-1}}^2 &\leq 2d_2k^2S_1^2(h)S^2(h)\sum_{m=2}^r|\uh{m-1}||A_h\uh{m-1}|^2\\
&\leq 2d_0d_2k^2S_1^2(h)S^2(h)\lambda_N\sum_{m-1}^r|A_h\uh{m-1}|^2\\
&\leq k(\nu-\delta)\sum_{m=1}^r|A_h\uh{m-1}|^2.
\endaligned
\end{equation}
We apply this upper bound into \pref{5pt49}, we get \pref{5pt52} for the integer $r$.
To complete the proof it suffices to show that  conditions $(i)$ and $(ii)$ in Lemma \ref{lemma-nsesi} ensure the condition \pref{5pt51}.
We recall that
since $\|\uh{0}\|\leq\|u_0\| $ for all $h$ and
$$k\sum_{m=1}^N|f^m|^2\leq \int_0^T |f(s)|^2\;ds,$$
then,
\begin{equation}
\aligned
\lambda_N &\leq \|u_0\|^2 + \frac{d_0^2}{\nu}\int_0^T||f(s)|^2\;ds + 2k^2S_1^2(h)S^2(h)|u_0|^2\Vnorm{\uh{0}}^2\\
&\leq d_{10} + 2k^2S_1^2(h)S^2(h)|u_0|^2\|u_0\|^2.
\endaligned
\end{equation}
Hence, if $kS_1^2S^2(h)\leq d'$, then 
\begin{equation}
2kd_0d_2S_1^2(h)S^2(h)\lambda_N \leq 2d'(d_{10}+2d'd''|u_0|^2\|u_0\|^2)
\leq \nu-\delta,
\end{equation}
provided $d',\; d''$ are sufficiently small.
\end{proof}

%% file: summary.tex
%%%%%%%%%%%%%%%%%%%%%%%%%%%%%%%%%%%%%%%%%%%%%%%%%%%%%%%%%%%%%%%%%%%%%%%%%%%%%%
\section{Summary}\label{conclusion}

 The NSV model of viscoelastic incompressible fluid has been proposed as a regularization of the 3D NSE for purposes of direct numerical simulation.  In this work, we have shown one of the benefits of using the 2D NSV turbulence model for small regularization parameter $\alpha$, instead of the 2D NSE to reduce computational expense when automating the inpainting process.  
To be more precise, one can find a parameter $\alpha>0$ in which the 2D NSV gives a solution to the image inpainting problem comparable (both using subjective and objective measure) to that of the solution of the 2D NSE but only requires a time step much larger in comparison to that of 2D NSE.  That is, the 2D NSV converge to the steady state solution with a much larger time step and hence solves the image inpainting problem using less computational resources.  In the numerical experiments, we found that after accounting for the relative costs of the two methods, the 2D NSV gives a solution to the inpainting problem which matches the quality of the image produced by when NSE is used but using less resources.

In future work we would like to investigate other PDEs which can be used instead of the 2D NSE and 2D NSV when solving the inpainting process.  In particular, we would like to use a PDE to solve the inpainting problem without the addition of anisotropic diffusion.  The dependence of the stability of solutions on certain characteristics of the image near the boundary is also of major interest in this topic.  We would like to investigate the dependence of the uniqueness of steady solution on the viscosity, on the image at the boundary, and on the size of the inpainting region.  We have presented some basic results on the uniqueness of steady state solution of the 2D steady state applied in the context of image inpainting.  We would like to do further numerical experiments to verify these hypothesis.

%% file: ack.tex
\section{Acknowledgments}

We would like to thank E.S. Titi, D. Reynolds, and Y. Zhou for valuable discussions on discretization techniques and other helpful comments on this study.
In particular, we would like to thank E.S. Titi for showing us the proof of the dependence of the uniqueness of the inpainting solution on the viscosity, on the image at the boundary and the size of inpainting region. 
MH was supported in part by NSF Awards~0715146 and 0915220.
ME and EL were supported in part by NSF Award~0715146.

%% file: m.bbl
\begin{thebibliography}{10}

\bibitem{AR75}
{\sc R.~Adams}, {\em Compact imbeddings of {$W^{m,p}(\Omega)$}}, in Sobolev
  Spaces, S.~{\ E}ilenberg and H.~Bass, eds., Academic Press, New York, 1975.

\bibitem{ATXX}
{\sc W.~Au and R.~Takei}, {\em Image inpainting with the {Navier-Stokes}
  equations}.
\newblock Available at Final Report APMA 930 SFU, 2002.

\bibitem{BLT07}
{\sc C.~Bardos, J.~Linshiz, and {\ E}.~S. Titi}, {\em Global regularity for a
  {Birkhoff}-{Rott}-$\alpha$ approximation of the dynamics of vortex sheets of
  the {2D} {{\ E}uler} equations}.
\newblock 2008.

\bibitem{BBS01}
{\sc M.~Bertalmio, A.~L. Bertozzi, and G.~Sapiro}, {\em {Navier-Stokes}, fluid
  dynamics, and image and video inpainting}, 2001 I{\ E}{\ E}{\ E} Computer
  Society Conference on Computer Vision and Pattern Recognition (CVPR'01), 1
  (2001), p.~35.

\bibitem{BSCB00}
{\sc M.~Bertalmio, G.~Sapiro, V.~Caselles, and C.~Ballester}, {\em Image
  inpainting}, International Conference on Computer Graphics and Interactive
  Techniques,  (2000), pp.~417--424.

\bibitem{BoMarz07}
{\sc F.~Bornemann and T.~Marz}, {\em Fast image inpainting based on coherence
  transport}, {J. Math, Imaging Vis.}, 28 (2007), pp.~259--278.

\bibitem{CHT05}
{\sc C.~Cao, D.~Holm, and {\ E}.~Titi}, {\em On the {Clark-$\alpha$} model of
  turbulence: global regularity and long-time dynamics}, Journal of Turbulence,
  6 (2005), pp.~1--11.

\bibitem{CLT06}
{\sc Y.~Cao, {\ E}.~Lunasin, and {\ E}.~Titi}, {\em Global well-posedness of
  the viscous and inviscid simplified {Bardina} model}, Communications in
  Mathematical Sciences, 4 (2006), pp.~823--848.

\bibitem{CFHOT98}
{\sc S.~Chen, C.~Foias, D.~Holm, {\ E}.~Olson, {\ E}.~Titi, and S.~Wynne}, {\em
  {Camassa--Holm} equations as closure model for turbulent channel and pipe
  flow}, Phys.\ Rev.\ Lett., 81 (1998), pp.~5338--5341.

\bibitem{CFHOT99a}
\leavevmode\vrule height 2pt depth -1.6pt width 23pt, {\em The {Camassa--Holm}
  equations and turbulence}, Phys. D, 133 (1999), pp.~49--65.

\bibitem{CFHOT99b}
\leavevmode\vrule height 2pt depth -1.6pt width 23pt, {\em A connection between
  the {Camassa-Holm} equations and turbulent flows in channels and pipes},
  Phys. Fluids, 11 (1999), pp.~2343--2353.

\bibitem{CHMZ99}
{\sc S.~Chen, D.~Holm, L.~Margolin, and R.~Zhang}, {\em Direct numerical
  simulation of the {Navier--Stokes} alpha model}, Phys. D, 133 (1999),
  pp.~66--83.

\bibitem{CHOT05}
{\sc A.~Cheskidov, D.~Holm, {\ E}.~Olson, and {\ E}.~Titi}, {\em On a
  {Leray-$\alpha$} model of turbulence}, Royal Soc. A, Mathematical, Physical
  and {\ E}ngineering Sciences, 461 (2005), pp.~629--649.

\bibitem{CF88}
{\sc P.~Constantin and C.~Foias}, {\em {\ E}xistence and uniqueness theorems},
  in Navier-Stokes equations, J.~P. May, R.~Zimmer, and S.~Block, eds., The
  University of Chicago Press, Chicago and London, 1988.

\bibitem{FHT01}
{\sc C.~Foias, D.~Holm, and {\ E}.~Titi}, {\em The {Navier--Stokes--alpha}
  model of fluid turbulence. advances in nonlinear mathematics and science},
  Phys. D, 152/153 (2001), pp.~505--519.

\bibitem{FHT02}
\leavevmode\vrule height 2pt depth -1.6pt width 23pt, {\em The
  three-dimensional viscous {Camassa--Holm} equations, and their relation to
  the {Navier--Stokes} equations and turbulence theory}, J. Dynam. Differential
  {\ E}quations, 14 (2002), pp.~1--35.

\bibitem{GH99}
{\sc B.~Geurts and D.~Holm}, {\em Fluctuation effects on 3d-{Lagrangian} mean
  and {{\ E}ulerian} mean fluid motion}, Phys. D, 133 (1999), pp.~215--269.

\bibitem{GH03}
\leavevmode\vrule height 2pt depth -1.6pt width 23pt, {\em Regularization
  modeling for large eddy simulation}, Phys. Fluids, 15 (2003), pp.~L13--L16.

\bibitem{HMR98}
{\sc D.~Holm, J.~Marsden, and T.~Ratiu}, {\em {{\ E}uler-Poincar\'{e}} models
  of ideal fluids with nonlinear dispersion}, Phys.\ Rev.\ Lett., 80 (1998),
  pp.~4173--4176.

\bibitem{ILT06}
{\sc A.~Ilyin, {\ E}.~Lunasin, and {\ E}.~Titi}, {\em A
  {modified-Leray-$\alpha$} subgrid scale model of turbulence}, Nonlinearity,
  19 (2006), pp.~879--897.

\bibitem{KT07}
{\sc V.~Kalantarov and {\ E}.~S. Titi}, {\em Global attractors and estimates of
  the number of degrees of freedom of determining modes for the {3D}
  {Navier-Stokes-Voight} equations}.
\newblock arXiv.0705.3972v1, 2007.

\bibitem{KLT07}
{\sc V.~K. Kalantarov, B.~Levant, and {\ E}.~Titi}, {\em Gevrey regularity of
  the global attractor of the {3D} {Navier-Stokes-Voight} equations}, J.
  Nonlinear Sci.,  (2007).
\newblock to appear.

\bibitem{KhTi07}
{\sc B.~Khouider and {\ E}.~S. Titi}, {\em An inviscid regularization for the
  surface quasi-geostrophic equation}, Comm. Pure Appl. Math, 61 (2008),
  pp.~1331--1346.

\bibitem{LL06}
{\sc W.~Layton and R.~Lewandowski}, {\em On a well-posed turbulence model},
  Discrete and Continuous Dyn. Sys. B, 6 (2006), pp.~111--128.

\bibitem{LiTi07}
{\sc J.~Linshiz and {\ E}.~Titi}, {\em Analytical study of certain
  magnetohydrodynamic-alpha models}, J. Math. Phys., 48 (2007).

\bibitem{LKTT07}
{\sc {\ E}.~Lunasin, S.~Kurien, M.~Taylor, and {\ E}.~Titi}, {\em A study of
  the {Navier--Stokes-$\alpha$} model for two-dimensional turbulence}, Journal
  of Turbulence, 8 (2007), pp.~751--778.

\bibitem{LKT08}
{\sc {\ E}.~Lunasin, S.~Kurien, and {\ E}.~Titi}, {\em Spectral scaling of the
  leray-$\alpha$ model for two-dimensional turbulence}, J. Phys. A: Math.
  Theor., 41 (2008), p.~344014.

\bibitem{MM03}
{\sc K.~Mohseni, B.~Kosovi\'{c}, S.~Shkoller, and J.~Marsden}, {\em Numerical
  simulations of the {Lagrangian} averaged {Navier-Stokes} equations for
  homogeneous isotropic turbulence}, Phys. Fluids, 15 (2003), pp.~524--544.

\bibitem{OlTi07}
{\sc {\ E}.~Olson and {\ E}.~Titi}, {\em Viscosity versus vorticity stretching:
  global well-posedness for a family of {Navier--Stokes-alpha-like} models},
  Nonlinear Anal., 66 (2007), pp.~2427--2458.

\bibitem{Osk73}
{\sc A.~P. Oskolkov}, {\em The uniqueness and solvability in the large of the
  boundary value problems for the equations of motion of aqueous solutions of
  polymers}, Zap. Naucn. Sem. Leningrad. Otdel. Mat. Inst. Steklov (LOMI), 38
  (1973), pp.~98--136.

\bibitem{Osk80}
\leavevmode\vrule height 2pt depth -1.6pt width 23pt, {\em On the theory of
  {Voight} fluids}, Zap. Naucn. Sem. Leningrad. Otdel. Mat. Inst. Steklov
  (LOMI), 96 (1980), pp.~233--236.

\bibitem{Temam88}
{\sc R.~Temam}, {\em Fluids driven by its boundary}, in Infinite-dimensional
  Dynamical Systems in Mechanics and Physics, F.~John, J.~Marsden, and
  L.~Sirovich, eds., Springer-Verlag, New York, NY, 1988, ch.~3, pp.~116--119.

\bibitem{Temam01}
\leavevmode\vrule height 2pt depth -1.6pt width 23pt, {\em {N}avier-{S}tokes
  equations: Theory and numerical analysis}, AMS Chelsea Publications,
  Providence, Rode Island, 2001.
\newblock ISBN: 0821827375.

\bibitem{TscDer05}
{\sc D.~Tschumperle and R.~Deriche}, {\em Vector-valued image regularization
  with {PDEs}: {A} common framework for different applications.}, {IEEE Trans.
  Pattern Anal. Mach. Intell.}, 27 (2005), pp.~506--517.

\bibitem{YXTK96}
{\sc Y.~You, W.~Xu, A.~Tannenbaum, and M.~Kaveh}, {\em Behavioral analysis of
  anisotropic diffusion in image processing}, I{\ E}{\ E}{\ E} Trans. on Image
  Processing, 5 (1996), pp.~1539--1552.

\end{thebibliography}
